\documentclass[12pt,reqno]{amsart}
\usepackage{amsmath,amsthm,amssymb,amsfonts,amscd}
\usepackage{mathrsfs}
\usepackage{bbm}
\usepackage{bbding}
\usepackage{hyperref}
\usepackage{geometry}
\usepackage{color}
\usepackage{xcolor}

\usepackage{picture,epic}
\usepackage{tikz}

\numberwithin{equation}{section}

\theoremstyle{plain}
\newtheorem{theorem}{Theorem}
\newtheorem{lemma}[theorem]{Lemma}

\newtheorem{proposition}[theorem]{Proposition}

\theoremstyle{definition}

\theoremstyle{remark}
\newtheorem{remark}[theorem]{Remark}

\renewcommand{\Re}{\operatorname{Re}}
\renewcommand{\Im}{\operatorname{Im}}

\newcommand{\sgn}{\operatorname{sgn}}

\newcommand{\supp}{\operatorname{supp}}

\newcommand{\Sym}{\operatorname{Sym}}

\newcommand{\GL}{\operatorname{GL}}
\newcommand{\SL}{\operatorname{SL}}

\renewcommand{\mod}{\operatorname{mod}\ }

\newcommand{\dd}{\mathrm{d}}
\newcommand{\Res}{\mathop{\operatorname{Res}}}

\makeatletter
\def\@tocline#1#2#3#4#5#6#7{\relax
  \ifnum #1>\c@tocdepth 
  \else
    \par \addpenalty\@secpenalty\addvspace{#2}%
    \begingroup \hyphenpenalty\@M
    \@ifempty{#4}{%
      \@tempdima\csname r@tocindent\number#1\endcsname\relax
    }{%
      \@tempdima#4\relax
    }%
    \parindent\z@ \leftskip#3\relax \advance\leftskip\@tempdima\relax
    \rightskip\@pnumwidth plus4em \parfillskip-\@pnumwidth
    #5\leavevmode\hskip-\@tempdima
      \ifcase #1
       \or\or \hskip 1em \or \hskip 2em \else \hskip 3em \fi%
      #6\nobreak\relax
    \hfill\hbox to\@pnumwidth{\@tocpagenum{#7}}\par
    \nobreak
    \endgroup
  \fi}
\makeatother

\begin{document}

\title
{On the Rankin--Selberg problem}
\author{Bingrong Huang}
\address{Data Science Institute and School of Mathematics \\ Shandong University \\ Jinan \\ Shandong 250100 \\China}
\email{brhuang@sdu.edu.cn}


\date{\today}

\begin{abstract}
  In this paper, we solve the Rankin--Selberg problem. That is, we break the well known Rankin--Selberg's bound on the error term of the second moment of Fourier coefficients of a $\GL(2)$ cusp form (both holomorphic and Maass), which remains its record since its birth for more than 80 years. We extend our method to deal with averages of coefficients of L-functions which can be factorized as a product of a degree one and a degree three L-functions.
\end{abstract}



\thanks{This work was supported by the Young Taishan Scholars Program of Shandong Province (Grant No. tsqn201909046), Qilu Young Scholar Program of Shandong University, and NSFC (Nos. 12001314 and 12031008).}

\maketitle

\section{Introduction} \label{sec:Intr}


Let $L(s,f)$ be an L-function of degree $d$ in the sense of Iwaniec--Kowalski \cite[\S5.1]{IwaniecKowalski2004analytic}
with coefficients $\lambda_f(1)=1$, $\lambda_f(n)\in\mathbb{C}$.
See \S\ref{subsec:AFE} below.
It is a fundamental problem to prove an asymptotic formula for the sum
\[
  \mathcal A(X,f) = \sum_{n\leq X} \lambda_f(n).
\]
Under some suitable conditions, one can prove an asymptotic formula
\[
  \mathcal A(X,f) = \Res_{s=1} \frac{L(s,f) X^s}{s} + O_f\left(X^{\frac{d-1}{d+1}+o(1)}\right),
\]
in quite general situations. See e.g. Friedlander--Iwaniec \cite{FI2005}.
If $d\geq4$, the generalized Riemann hypothesis (GRH) for $L(s,f)$ implies the exponent $\frac{d-1}{d+1}$ can be replaced by $1/2$.
This type of result has a very important application to the generalized Ramanujan conjecture (see Serre \cite{Serre}).
In this paper, we consider an important special case when $f$ is the Rankin--Selberg convolution of a $\GL(2)$ automorphic representation with itself, 
which has degree $d=4$.
Our goal is to beat the exponent $3/5$ in the error term.

Let $\phi$ be a $\GL(2)$ holomorphic Hecke cusp form or
Hecke--Maass cusp form for $\SL(2,\mathbb{Z})$.
Let $\lambda_\phi(n)$ be its $n$-th Hecke eigenvalue. We define
\[
  \mathcal S_2(X,\phi) = \sum_{n\leq X} \lambda_\phi(n)^2,\quad
  \Delta_2(X,\phi) = \mathcal S_2(X,\phi) - c_\phi X,
\]
where $c_\phi = L(1,\Sym^2 \phi)/\zeta(2)$ and $L(s,\Sym^2 \phi)$ is the symmetric square L-function of $\phi$.
Rankin \cite{Rankin1939} and Selberg  \cite{Selberg1940} invented the powerful Rankin--Selberg method, and then successfully showed that
\[
  \Delta_2(X,\phi) \ll_\phi X^{3/5}.
\]
This bound remains the best since it was proved more than 80 years.
The \emph{Rankin--Selberg problem} is to improve the exponent $3/5$.
Although we have several methods to prove essentially the same bound as above (see e.g. Ivi\'{c} \cite{Ivic2008}), the exponent $3/5$ represents one of the longest standing records in analytic number theory.
The generalized Riemann Hypothesis implies $\Delta_2(X,\phi) \ll X^{1/2+o(1)}$.
It is conjectured  that (see e.g. Ivi\'{c} \cite[Eq. (7.23)]{Ivic1990large})
\[
  \Delta_2(X,\phi) = O_\phi(X^{3/8+o(1)})
  \quad \textrm{and} \quad
  \Delta_2(X,\phi) = \Omega_\phi(X^{3/8}).
\]
The above $\Omega$-result was proved by Lau--L\"{u}--Wu \cite{LauLuWu2011integral}.
Note that $\mathcal S_2(X,\phi)$ is essentially the same as $\mathcal A(X,f)$ with $f=\phi\times\phi$ in which case we have the degree $d=4$.
In this paper, our main goal is to solve the Rankin--Selberg problem.

\begin{theorem}\label{thm:RS}
  With the notation as above. We have
  \begin{equation}\label{eqn:Delta<<}
    \Delta_2(X,\phi) \ll_\phi X^{3/5-\delta+o(1)},
  \end{equation}
  for any $\delta\leq 1/560 =0.001785...$.
\end{theorem}

%

\begin{remark}
  We emphasize that we do not expect our bounds to be optimal by our method. For example one may find a better exponent pair to improve our exponent, see e.g. Graham--Kolesnik \cite[Chap. 5 and 7]{GK}. Let $(k,\ell)$ be an exponent pair (see \cite[Chap. 3]{GK}). To get a better bound, we essentially need to minimize $(57 + 52 k - 42 \ell)/(97 + 82 k - 72 \ell)$.
  We choose the exponent pair $(1/30,13/15)$ obtained by the simplest van der Corput estimates (with fifth derivative), which gives a rather good bound and can be extended to more general cases (see Friedlander--Iwaniec \cite[\S4]{FI2005} and Remark \ref{remark:chi} below).
  We remark that  the well known exponent pair $(9/56+\varepsilon,37/56+\varepsilon)$ is not good for our purpose, but this combines with $A$-process twice we can take $A^2(9/56+\varepsilon,37/56+\varepsilon)=(9/278+\varepsilon, 241/278+\varepsilon)$ which may allow us to take $\delta=6/3235=0.001854...$.
  The best possible $\delta$ we may show is $37/19220=0.001925...$ by using the exponent pair $(13/414+\varepsilon, 359/414+\varepsilon)$, which is obtained by using  Bourgain's exponent pair $(13/84+\varepsilon,55/84+\varepsilon)$ (see \cite{Bourgain2017decoupling}) and $A$-process twice, i.e., $(13/414+\varepsilon, 359/414+\varepsilon)= A^2(13/84+\varepsilon,55/84+\varepsilon)$.
\end{remark}

\begin{remark}
  The same method works for $\phi$ being either holomorphic or Maass. In fact, the holomorphic case is easier, since we have the Ramanujan bounds for the coefficients, so we will give the proof for the Maass case.
\end{remark}

\begin{remark}
  If $\phi$ is a dihedral Maass cusp form (also a CM holomorphic cusp form), then we can prove
  \[
    \Delta_2(X,\phi) \ll_\phi X^{1/2+o(1)},
  \]
  unconditionally.
  The key fact we need is the factorization of the symmetric square L-function of a dihedral form. Then the result will follow from the approximate functional equations, the Cauchy--Schwarz inequality, and the integral mean-value estimate \eqref{eqn:IMV}.
\end{remark}

To prove Theorem \ref{thm:RS}, we first consider $\mathcal A(X,\phi\times \phi)$.
Let
\begin{equation}\label{eqn:RS=1+Sym2}
 L(s,\phi\times \phi) = \zeta(2s) \sum_{n=1}^{\infty} \frac{\lambda_\phi(n)^2}{n^s}, \quad \Re(s)>1
\end{equation}
be the Rankin--Selberg L-function of $\phi\times \phi$.
Note that $\phi\times \phi = 1 \boxplus \Sym^2 \phi$, where $\Sym^2 \phi$ is the symmetric square lift of $\phi$. That is $L(s,\phi\times \phi) = \zeta(s)L(s,\Sym^2 \phi)$.
In \cite{GelbartJacquet}, Gelbart--Jacquet proved that $\Sym^2 \phi$ is an automorphic cuspidal representation for $\GL(3)$.

This leads us to consider the more general case
\[
  \textrm{$f=1\boxplus g$, \;\; that is, \;\; $L(s,f)=\zeta(s)L(s,g)$},
\]
where $L(s,g)$ is a primitive L-function of degree $d-1$, that is, $L(s,g)$  cannot be decomposed into a product of L-functions of lower degrees.
When $d=2$ and $L(s,g)=\zeta(s)$, this is  the well-known classical divisor problem, in which case we can do better than the  exponent $1/3$ by using the theory of exponential sums (see e.g. Titchmarsh \cite[\S12.4]{titchmarsh1986theory}).
When $d=3$ and $L(s,g)=L(s,\phi)$, Friedlander--Iwaniec \cite[\S 4]{FI2005} showed that one can beat the exponent $1/2$ (see more discussion in Remark \ref{remark:d=3}). In this paper, we deal with the case  when $d=4$.

Let $\Phi$ be a Hecke--Maass cusp form for $\SL(3,\mathbb Z)$.
Let $A_\Phi(1,n)$ be the normalized Fourier coefficients of $\Phi$. The \emph{generalized Ramanujan conjecture} (GRC) for $\Phi$ asserts that $A_\Phi(1,n) \ll n^{o(1)}$.
Our method can be used to prove the following result under GRC.

\begin{theorem}\label{thm:1+3}
  With the notation as above. Assuming GRC for $\Phi$, then we have
  \begin{equation}\label{eqn:A=}
    \mathcal A(X,1\boxplus\Phi)  = L(1,\Phi) \, X + O_\Phi( X^{3/5-\delta+o(1)} ),
  \end{equation}
  for any $\delta\leq 1/560$. Furthermore, if $\Phi=\Sym^2 \phi$, then we don't need to assume GRC for $\Phi$.
\end{theorem}

\begin{remark}
  Let $\phi$ be a $\GL(2)$ Hecke--Maass cusp form for $\SL(2,\mathbb{Z})$. The \emph{Ramanujan conjecture} (RC) for $\phi$ says that $\lambda_\phi(n)\ll n^{o(1)}$.  In the above theorem we don't need to assume RC for $\phi$ (hence GRC for $\Sym^2 \phi$). The reason is that we have nonnegativity of the coefficients
  \begin{equation}\label{eqn:nonnegativity}
    \lambda_{1\boxplus\Sym^2\phi}(n) = \lambda_{\phi\times\phi}(n) = \sum_{\ell^2 m=n} \lambda_\phi(m)^2 \geq 0,
  \end{equation}
  by \eqref{eqn:RS=1+Sym2}. See more details in \S \ref{sec:dual}.
\end{remark}

\begin{remark}\label{remark:chi}
  As in Friedlander--Iwaniec \cite[\S 4]{FI2005}, we can prove similar result when we replace $1$ by a Dirichlet character $\chi$, that is, when $f=\chi\boxplus \Phi$.
\end{remark}

\begin{remark}\label{remark:d=3}
  In the case $f=1\boxplus \phi$ with degree $d=3$, under RC for $\phi$, the general theorem  (see e.g. \cite{FI2005}) will give us
  \[
    \mathcal A(X,1\boxplus \phi) = L(1,\phi) \, X + O_\phi(X^{1/2+o(1)}).
  \]
  A simple application of GRH will still give us the exponent $1/2$. Our method in this paper can be applied to this case, and we can show
  \[
    \mathcal A(X,1\boxplus \phi) = L(1,\phi) \, X + O_\phi(X^{1/2-\delta'+o(1)}),
  \]
  for some small positive $\delta'$, under RC.
  This slightly goes beyond a simple application of GRH.
  The same estimate holds for a holomorphic cusp form $\phi$. Note that in the holomorphic case, RC is known.
\end{remark}

There are two different methods to get $\frac{d-1}{d+1}$ if $f=1\boxplus g$ with $g$ primitive and $d\geq3$ as mentioned in Friedlander--Iwaniec \cite{FI2005}. We use the one with the contour of integral on the critical line. It is also possible to avoid this by shifting the contour to a vertical line with negative real part as done by Friedlander--Iwaniec.
Our approach leads to a new integral moment of L-functions
\[
  \int_{T}^{2T} L(1/2+it,1\boxplus\Phi) X^{it} \dd t.
\]
Finding good upper bounds for this integral moment is of independent interest and we want to highlight (see Proposition \ref{prop:I}).
By the approximate functional equation, the Cauchy inequality, and the integral mean value estimate, one can show the upper bound $O(T^{5/4+o(1)})$, which is good enough for small $T$'s. The most important case is when $T=X^{2/5+\delta}$. We seek for a better upper bound.
Our idea is to use moments of L-functions without absolute value, which reduces the problem to a dual sum of Fourier coefficients (cf. Huang \cite[\S7]{Huang}).

In order to prove Theorem  \ref{thm:1+3}, we will use a power saving for the analytic twisted sum of $\GL(3)$ Fourier coefficients.
Define
\[
  \mathscr{S}(N) := \sum_{n\geq1} A_\Phi(1,n) e\left(T \varphi\left(\frac{n}{N}\right)\right) V\left(\frac{n}{N}\right),
\]
where $T\geq1$ is a large parameter, $\varphi$ is some fixed real-valued smooth function, and $V\in C_c^\infty (\mathbb{R})$ with $\supp V\subset [1/2,1]$, total variation $\mathrm{Var}(V)\ll 1$ and satisfying that $V^{(k)} \ll P^k$ for all $k\geq0$ with $P\ll T^{\eta}$ for some small $\eta\in [0,1/10]$.

\begin{theorem}\label{thm:3}
  Assume $\varphi(u)=u^\beta$ with $\beta\in (0,1)$.
  Then we have
  \[
    \mathscr S(N) \ll_\Phi T^{3/10} N^{3/4+\varepsilon},
  \]
  if $T^{6/5} \leq N \leq T^{8/5-\varepsilon}$, and we have
  \[
    \mathscr S(N) \ll_\Phi T^{-1/2} N^{5/4+\varepsilon},
  \]
  if $T^{8/5-\varepsilon} \leq N \leq T^{2}$.
\end{theorem}

In \cite{Munshi2015circleIII}, Munshi proved the first nontrivial result of this type for $\varphi(u)=-(\log u)/2\pi$ with $N\leq T^{3/2+\varepsilon}$, and got an application to the subconvexity bounds of $\GL(3)$ L-functions in the $T$-aspect. Recently, this was strengthened to the above bound for $\varphi(u)=-(\log u)/2\pi$ and $N\leq T^{3/2+\varepsilon}$ by Aggarwal \cite{Aggarwal} and for $\varphi(u)=u^{\beta}$ and $T=\alpha N^\beta$ by Kumar--Mallesham--Singh \cite{KMS} (with bounds depending on $\alpha$).
However, for Theorem  \ref{thm:1+3}, we need $\alpha$ to be quite large.
We also need the result for $N\geq T^{3/2+\varepsilon}$, which is unlike the subconvexity problem for $L(1/2+iT,\Phi)$. We will modify (and simplify) their methods to prove Theorem \ref{thm:3}. In fact, we use the Duke--Friedlander--Iwaniec delta method similar to what Munshi \cite{munshi2018} did, instead of the Kloosterman circle method.
We can deal with more general $\varphi$'s. For our main application, we only need $\varphi(u)=u^{1/4}$.

Similar result for $\GL(2)$ Fourier coefficients can be found in Jutila \cite{Jutila}. In \cite{munshi2018}, Munshi showed the first nontrivial result of this type for $\GL(3)\times \GL(2)$ Fourier coefficients with $\varphi(u)=-(\log u)/2\pi$. Recently, Lin--Sun \cite{LinSun} succeeded to treat the analytic twisted sum of $\GL(3)\times \GL(2)$ Fourier coefficients, and got an application to $\mathcal A(X,f)$ with $f=\Phi\times \phi$ under GRC. Here we follow Lin--Sun's formulation. See the introduction there for more discussions on this topics.

\begin{remark}
  In the case $f=1\boxplus (\Phi\times \phi)$ with degree $d=7$, under GRC for $\Phi$ and  $\phi$, the general theorem (see e.g. \cite{FI2005}) will give us
  \[
    \mathcal A(X,1\boxplus (\Phi\times \phi)) = L(1,\Phi\times \phi) \, X
    + O_{\Phi,\phi}(X^{3/4+o(1)}).
  \]
  Our method in this paper can also be applied to this case, and we can show
  \[
    \mathcal A(X,1\boxplus (\Phi\times \phi)) = L(1,\Phi\times \phi) \, X + O_{\Phi,\phi}(X^{3/4-\delta''+o(1)}),
  \]
  for some small positive $\delta''$, under GRC. To prove this, one needs to extend \cite[Theorem 1.1]{LinSun} to the case $N\geq t^{3+\varepsilon}$. The proof is similar (but more complicated), so we don't include it in this paper.
\end{remark}


%
%
%
%
%


\medskip
The plan of this paper is as follows.
In \S\ref{sec:preliminaries}, we recall some results on L-functions, the Voronoi summation formula, and the Duke--Friedlander--Iwaniec
delta method.
In \S\ref{sec:dual}, we apply smoothing, Mellin transform and the stationary phase to reduce to the dual sums. This gives a proof of Theorem \ref{thm:1+3} by assuming Theorem \ref{thm:3}.
Then we prove Theorem \ref{thm:RS} in \S \ref{sec:proofRS}.
Finally, in \S\ref{sec:proof3}, we proof Theorem \ref{thm:3}.
The method is relatively standard now thanks to Munshi and his collaborators. In \S\ref{subsec:delta}, we apply the delta method. In \S\ref{subsec:summation_formula}, we use the summation formulas to get the dual sums.
We first need to analyse the integrals in \S\ref{subsec:integrals}, and then apply Cauchy and Poisson in the generic case in \S\ref{subsec:Cauchy} and \S\ref{subsec:Poisson}.

\medskip
\textbf{Notation.}
Throughout the paper, $\varepsilon$ is an arbitrarily small positive number;
all of them may be different at each occurrence.
The weight functions $U,\ V,\ W$ may also change at each occurrence.
As usual, $e(x)=e^{2\pi i x}$.

\section{Preliminaries}\label{sec:preliminaries}

\subsection{Maass forms and L-functions} \label{subsec:L-functions}

Let $\Phi$ be a Hecke--Maass form of type $(\nu_1,\nu_2)$ for $\SL(3,\mathbb{Z})$ with the normalized Fourier coefficients $A(m,n)$ such that $A(1,1)=1$.
The Langlands parameters are defined as $\alpha_1 = -\nu_1 - 2\nu_2 +1$,
$\alpha_2 = -\nu_1 +\nu_2$, and $\alpha_3 = 2\nu_1 +\nu_2 -1$.
The Ramanujan--Selberg conjecture predicts that $\Re(\alpha_i) = 0$.
From the work of Jacquet and Shalika \cite{JacquetShalika}, we know (at least) that $|\Re(\alpha_i)| < 1/2$.
It is well known that by standard properties of the Rankin--Selberg L-function we have the Ramanujan conjecture on average
\begin{equation}\label{eqn:RS3}
  \mathop{\sum_{m\geq1}\sum_{n\geq1}}_{m^2 n \leq N} |A(m,n)|^2
  \ll_\Phi N^{1+\varepsilon}.
\end{equation}
The L-function associated with $\Phi$ is given by
$L(s,\Phi) = \sum_{n=1}^{\infty} A(1,n) n^{-s}$ in the domain $\Re(s)>1$. It extends to an entire function and satisfies the following functional equation
\[
  \gamma(s,\Phi) L(s,\Phi) = \gamma(1-s,\tilde\Phi) L(1-s,\tilde{\Phi}), \quad \gamma(s,\Phi) = \prod_{j=1}^{3} \pi^{-s/2} \Gamma\left(\frac{s-\alpha_j}{2}\right).
\]
Here $\tilde\Phi$ is the dual form having Langlands parameters $(-\alpha_3,-\alpha_2,-\alpha_1)$ and the Fourier coefficients $\overline{A(m,n)}$. See more information in Goldfeld \cite{goldfeld2006automorphic}.
%

\subsection{Approximate functional equation}\label{subsec:AFE}

Let $L(s,f)$ be an L-function of degree $d$ in the sense of \cite[\S5.1]{IwaniecKowalski2004analytic}. More precisely, we have
\[
  L(s,f) = \sum_{n\geq1} \frac{\lambda_f(n)}{n^s} = \prod_p \prod_{j=1}^d \left(1-\frac{\alpha_j(p)}{p^s}\right)^{-1},
\]
with $\lambda_f(1)=1$, $\lambda_f(n)\in\mathbb{C}$, $\alpha_j(p)\in\mathbb{C}$, such that the series and Euler products are absolutely convergent for $\Re(s)>1$.
The sequence $\{\lambda_f(n)\}_{n\geq1}$ are called coefficients of $L(s,f)$. We assume $|\alpha_j(p)|<p$ for all $p$.
There exist an integer $q(f)\geq1$ and a gamma factor
\[
  \gamma(s,f) = \pi^{-ds/2} \prod_{j=1}^{d} \Gamma\left(\frac{s-\kappa_j}{2}\right)
\]
with $\kappa_j\in\mathbb{C}$ and $\Re(\kappa_j)<1/2$ such that the complete L-function
\[
  \Lambda(s,f) = q(f)^{s/2} \gamma(s,f) L(s,f)
\]
admits an analytic continuation to a meromorphic function for $s\in\mathbb{C}$ of order 1 with poles at most at $s=0$ and $s=1$. Moreover we assume $L(s,f)$ satisfies the functional equation
\[
  \Lambda(s,f) = \varepsilon(f) \Lambda(1-s,\bar f),
\]
where $\bar f$ is the dual of $f$ for which $\lambda_{\bar{f}}(n) = \overline{\lambda_f(n)}$, $\gamma(s,\bar f)=\overline{\gamma(\bar s,f)}$, $q(\bar f)=q(f)$, and $\varepsilon(f)$ is the root number of $L(s,f)$ satisfying that $|\varepsilon(f)|=1$.
We further assume that $\lambda_f(n)$'s satisfy the Ramanujan bound on average, that is,
\begin{equation}\label{eqn:averagebound}
  \sum_{n\leq x} |\lambda_f(n)| \ll x^{1+\varepsilon}.
\end{equation}

As $|t|\to\infty$, Stirling's formula gives
\begin{align*}
  \Gamma(\sigma+it) = \sqrt{2\pi}  |t|^{\sigma-\frac12+it}
  \exp\big( -\tfrac{\pi}{2}|t| - it + i\text{sgn}(t)  \tfrac{\pi}{2}(\sigma-\tfrac12)\big)
  \big(1 +O(|t|^{-1}) \big).
\end{align*}
Hence for $|t|\in[T,2T]$ with $T$ large, and $\varepsilon\leq \Re(w)\ll 1 $, $|\Im(w)|\ll T^\varepsilon$, and $\kappa_j\ll1$, we have
\begin{equation}\label{eqn:Stirling}
  \begin{split}
  \frac{\Gamma(\frac{1/2+it+w-\kappa_j}{2})}{\Gamma(\frac{1/2+it-\kappa_j}{2})} &= \left(\frac{|t|}{2}\right)^{w/2} e^{i\pi w/2} \big(1+O(T^{-1}) \big), \\
  \frac{\Gamma(\frac{1/2-it-\overline{\kappa_j}}{2})}{\Gamma(\frac{1/2+it-\kappa_j}{2})} &= \left(\frac{|t|}{2e}\right)^{-it} \big(1+O(T^{-1}) \big).
\end{split}
\end{equation}
For $t\in[T,2T]$, by the approximate functional equation \cite[\S5.2]{IwaniecKowalski2004analytic} we have
\begin{multline*}
  L(1/2+it,f)  =  \sum_{n\geq1} \frac{\lambda_f(n)}{n^{1/2+it}} \frac{1}{2\pi i} \int_{(2)} \frac{\gamma(1/2+it+w,f)}{\gamma(1/2+it,w)} \frac{q^{w/2}}{n^w} \frac{G(w)}{w} \dd w \\
  + \varepsilon(f) q^{-it} \sum_{n\geq1} \frac{\overline{\lambda_f(n)}}{n^{1/2-it}} \frac{1}{2\pi i} \int_{(2)} \frac{\gamma(1/2-it+w,\bar f)}{\gamma(1/2+it,f)} \frac{q^{w/2}}{n^w} \frac{G(w)}{w} \dd w + O(T^{-A}),
\end{multline*}
where $G(w)=e^{w^2}$.
We can move the line of integration to $\Re(w)=\varepsilon$ and truncate at $|\Im(s)|\leq T^{\varepsilon}$ with a negligible error term.
By Stirling's formula \eqref{eqn:averagebound} and \eqref{eqn:Stirling}, we can truncate the $n$-sum at $n\ll T^{d/2+\varepsilon}$ for the first sum and at $n\ll T^{d/2+\varepsilon}$ for the second sum above with a negligible error.
Hence by a smooth partition of unity, we prove the following lemma.

\begin{lemma}\label{lemma:AFE}
With notation as above, we have
\begin{multline}\label{eqn:AFE}
  L(1/2+it,f) = \frac{1}{2\pi i} \int_{\varepsilon-i T^{\varepsilon}}^{\varepsilon+iT^\varepsilon} \sum_{\substack{N\leq T^{d/2+\varepsilon}\\ N \; \textrm{dyadic}}} \sum_{n\geq1} \frac{\lambda_f(n)}{n^{1/2+it+w}} V\left(\frac{n}{N}\right)  \left(\frac{\pi t}{2}\right)^{dw/2} e^{i\pi dw/2} q^{w/2} \frac{G(w)}{w} \dd w  \\
  + \frac{\varepsilon(f)}{2\pi i} \int_{\varepsilon-i T^{\varepsilon}}^{\varepsilon+iT^\varepsilon} \sum_{\substack{N\leq T^{d/2+\varepsilon}\\ N \; \textrm{dyadic}}} \sum_{n\geq1}  \frac{\overline{\lambda_f(n)}}{n^{1/2-it+w}}  V\left(\frac{n}{N}\right)  \left(\frac{t}{2\pi e}\right)^{-idt}  \left(\frac{t}{2\pi}\right)^{dw/2} e^{i\pi dw/2} q^{w/2-it} \frac{G(w)}{w} \dd w  \\
  + O\left( T^{d/4-1+\varepsilon}\right),
\end{multline}
where $V$ is a fixed smooth function with $\supp V \subset (1/2,1)$.
\end{lemma}

\begin{remark}
  We can obtain a better error term in the above approximate functional equation by using Stirling's formula with better error term. Since \eqref{eqn:AFE} is good enough for our purpose in this paper, we don't do it here.
\end{remark}

\subsection{Voronoi summation formula} \label{subsec:VSF}

Let $\psi$ be a smooth compactly supported function on $(0,\infty)$,
and let $\tilde{\psi}(s):=\int_{0}^{\infty}\psi(x)x^s\frac{\dd x}{x}$
be its Mellin transform.
For $\sigma>5/14$, we define
\begin{equation}\label{eqn:Psi}
    \Psi^{\pm}(z) := z \frac{1}{2\pi i} \int_{(\sigma)} (\pi^3z)^{-s} \gamma^\pm(s) \tilde{\psi}(1-s)\dd s ,
\end{equation}
with
\begin{equation}\label{eqn:gamma^pm}
  \gamma^\pm(s) := \prod_{j=1}^{3}
    \frac{\Gamma\left(\frac{s+\alpha_j}{2}\right)} {\Gamma\left(\frac{1-s-\alpha_j}{2}\right)}
    \pm \frac{1}{i} \prod_{j=1}^{3}
    \frac{\Gamma\left(\frac{1+s+\alpha_j}{2}\right)} {\Gamma\left(\frac{2-s-\alpha_j}{2}\right)},
\end{equation}
where $\alpha_j$ are the Langlands parameters of $\phi$ as above.
Note that changing $\psi(y)$ to $\psi(y/N)$ for a positive real number $N$ has the effect of
changing $\Psi^\pm(z)$ to $\Psi^\pm(zN)$.
The Voronoi formula on $\GL(3)$ was first proved by Miller--Schmid~\cite{miller2006automorphic}.
The present version is due to Goldfeld--Li~\cite{goldfeld2006voronoi} with slightly renormalized variables (see Blomer \cite[Lemma 3]{blomer2012subconvexity}).
\begin{lemma}\label{lemma:VSF}
  Let $c,d,\bar{d}\in\mathbb Z$ with $c\neq0$, $(c,d)=1$, and $d\bar{d}\equiv1\pmod{c}$.
  Then we have
  \begin{equation*}
    \begin{split}
         \sum_{n=1}^{\infty} A(1,n)e\left(\frac{n\bar{d}}{c}\right)\psi(n)
         = \frac{c\pi^{3/2}}{2} \sum_{\pm} \sum_{n_1|c} \sum_{n_2=1}^{\infty}
              \frac{A(n_2,n_1)}{n_1n_2} S\left(d,\pm n_2;\frac{c}{n_1}\right)
              \Psi^{\pm}\left(\frac{n_1^2n_2}{c^3}\right),
    \end{split}
  \end{equation*}
  where $S(a,b;c) := \mathop{{\sum}^*}_{d(c)} e\left(\frac{ad+b\bar{d}}{c}\right)$ is the classical Kloosterman sum.
\end{lemma}

The function $\Psi^{\pm}(y)$  has the following properties.

\begin{lemma}\label{lemma:Psi=M+O}
  Suppose $\psi(y)$ is a smooth function, compactly supported on $[N,2N]$.
  Let $\Psi^\pm(z)$ be defined as in \eqref{eqn:Psi}.
  Then for any fixed integer $L\geq1$, and $zN\gg1$, we have
  \[
    \Psi^\pm(z) = z\int_0^\infty \psi(y) \sum_{\ell=1}^{L} \frac{\gamma_\ell}{(zy)^{\ell/3}}
    e\left(\pm3(zy)^{1/3}\right) \dd y + O\left((zN)^{1-L/3}\right),
  \]
  where $\gamma_\ell$ are constants depending only
  on $\alpha_1,\alpha_2,\alpha_3$, and $L$.
\end{lemma}

\begin{proof}
  See Li~\cite[Lemma 6.1]{li2009central} and Blomer~\cite[Lemma 6]{blomer2012subconvexity}.
\end{proof}


\subsection{The delta method}
There are two oscillatory factors contributing to the convolution sums. Our method is based on separating these oscillations using the circle method. In the present situation we will use a version of the delta method of Duke, Friedlander and Iwaniec. More specifically we will use the expansion (20.157) given in  \cite[\S20.5]{IwaniecKowalski2004analytic}. Let $\delta:\mathbb{Z}\rightarrow \{0,1\}$ be defined by
$$
\delta(n)=\begin{cases} 1&\text{if}\;\;n=0;\\
0&\text{otherwise}.\end{cases}
$$
We seek a Fourier expansion which matches with $\delta(n)$.
\begin{lemma}\label{lemma:delta}
  Let $Q$ be a large positive number. Then we have
  \begin{align}\label{eqn:delta-n}
    \delta(n)=\frac{1}{Q}\sum_{1\leq q\leq Q} \;\frac{1}{q}\; \sideset{}{^\star}\sum_{a\bmod{q}}e\left(\frac{na}{q}\right)
    \int_\mathbb{R} g(q,x) e\left(\frac{nx}{qQ}\right)\mathrm{d}x,
  \end{align}
  where 
  $g(q,x)$ is a weight function satisfies that
  \begin{equation}\label{eqn:g-h}
    g(q,x)=1+O\left(\frac{Q}{q}\left(\frac{q}{Q}+|x|\right)^A\right),
    \quad
     g(q,x)\ll |x|^{-A}, \quad \textrm{for any $A>1$},
  \end{equation}
  and
  \begin{equation}\label{eqn:g^(j)}
    \frac{\partial^j}{\partial x^j} g(q,x) \ll  |x|^{-j} \min(|x|^{-1},Q/q) \log Q, \quad j\geq1.
  \end{equation}
  Here the $\star$ on the sum indicates that the sum over $a$ is restricted by the condition $(a,q)=1$.
\end{lemma}

\begin{proof}
By \cite[eq. (20.157)]{IwaniecKowalski2004analytic}, we have
\begin{align*}
\delta(n)=\frac{1}{Q}\sum_{1\leq q\leq Q} \;\frac{1}{q}\; \sideset{}{^\star}\sum_{a\bmod{q}}e\left(\frac{na}{q}\right) \int_\mathbb{R}g(q,x) e\left(\frac{nx}{qQ}\right)\mathrm{d}x,
\end{align*}
for $n\in\mathbb Z$.
Here
\[
  g(q,x) = \int_{\mathbb{R}} \Delta_q(u) f(u) e\left(-\frac{ux}{qQ}\right) \dd u,
\]
with a smooth $f$ such that $\supp f\in[-Q^2/2,Q^2/2]$ and $f^{(j)}(u)\ll Q^{-2j}$, $j\geq0$,  and a function $\Delta_q(u)$ satisfies that
\cite[Lemma 20.17]{IwaniecKowalski2004analytic}
\begin{equation}\label{eqn:Delta_q}
  \Delta_q(u) \ll \frac{1}{(q+Q)Q} + \frac{1}{|u|+qQ},
  \quad
  \Delta_q^{(j)}(u) \ll \frac{1}{qQ} (|u|+qQ)^{-j}, \quad j\geq1.
\end{equation}
Note that here we take $N$ in \cite{IwaniecKowalski2004analytic} to be $Q^2/4$.
We recall the following two properties (see (20.158) and (20.159) of \cite{IwaniecKowalski2004analytic})
\footnote{There is a typo in \cite[eq. (20.158)]{IwaniecKowalski2004analytic}.}
\begin{equation*}
g(q,x) =1+h(q,x),\;\;\;\text{with}\;\;\;h(q,x) =O\left(\frac{Q}{q}\left(\frac{q}{Q}+|x|\right)^A\right),
\quad
 g(q,x)\ll |x|^{-A}
\end{equation*}
for any $A>1$.
In particular the second property implies that the effective range of the integral in \eqref{eqn:delta-n} is $[-Q^\varepsilon, Q^\varepsilon]$.

We will also need bounds of the derivatives of $g(q,x)$ with respect to $x$.
Note that
\[
  \frac{\partial^j}{\partial x^j} g(q,x) = \left(\frac{-2\pi i}{qQ}\right)^j  \int_{\mathbb{R}} \Delta_q(u) f(u) u^j e\left(-\frac{ux}{qQ}\right) \dd u
\]
By \eqref{eqn:Delta_q} and by $j$ and $j+1$ times repeated integration by parts, we have
\begin{equation*}
  \frac{\partial^j}{\partial x^j} g(q,x) \ll \min( |x|^{-j}  Q/q,  |x|^{-j-1} \log Q/q )
  \ll  |x|^{-j} \min(|x|^{-1},Q/q) \log Q.
\end{equation*}
This completes the proof of Lemma \ref{lemma:delta}.
\end{proof}

In applications of \eqref{eqn:delta-n}, we can first restrict to $|x|\ll Q^\varepsilon$. If $q\gg Q^{1-\varepsilon}$, then by \eqref{eqn:g^(j)} we get $ \frac{\partial^j}{\partial x^j} g(q,x) \ll Q^\varepsilon |x|^{-j} $, for any $j\geq1$. If $q\ll Q^{1-\varepsilon}$ and $Q^{-\varepsilon} \ll |x| \ll Q^\varepsilon$, then by \eqref{eqn:g^(j)} we also have $ \frac{\partial^j}{\partial x^j} g(q,x) \ll Q^\varepsilon |x|^{-j} $, for any $j\geq1$. Finally, if $q\ll Q^{1-\varepsilon}$ and $|x| \ll Q^{-\varepsilon}$, then by \eqref{eqn:g-h}, we get replace $g(q,x)$ by 1 with a negligible error term.
So in all cases, we can view $g(q,x)$ as a nice weight function.

\begin{remark}
  We can further prove
  \begin{align}\label{eqn:delta}
    \delta(n)=\frac{1}{Q}\sum_{1\leq q\leq Q} \;\frac{1}{q}\; \sideset{}{^\star}\sum_{a\bmod{q}}e\left(\frac{na}{q}\right) \int_\mathbb{R} w(q,x) e\left(\frac{nx}{qQ}\right)\mathrm{d}x + O(Q^{-A}),
  \end{align}
  where $w(q,x)$ is a weight function such that $\supp w(q,\cdot) \subset [-Q^\varepsilon,Q^\varepsilon]$ and  $\frac{\partial^j}{\partial x^j} w(q,x) \ll_j Q^\varepsilon (|x|^{-j}+Q^{j\varepsilon})$.
\end{remark}

\subsection{Oscillatory integrals}

Let $\mathcal{F}$ be an index set and $X=X_T:\mathcal{F}\rightarrow \mathbb{R}_{\geq1}$ be a function of $T\in\mathcal{F}$. A family of $\{w_T\}_{T\in\mathcal{F}}$ of smooth functions supported on a product of dyadic intervals in $\mathbb{R}_{>0}^d$ is called $X$-inert if for each $j=(j_1,\ldots,j_d) \in \mathbb{Z}_{\geq0}^d$ we have
\[
  \sup_{T\in\mathcal{F}} \sup_{(x_1,\ldots,x_d) \in \mathbb{R}_{>0}^d}
  X_T^{-j_1-\cdots -j_d} \left| x_1^{j_1} \cdots x_d^{j_d} w_T^{(j_1,\ldots,j_d)} (x_1,\ldots,x_d) \right|
   \ll_{j_1,\ldots,j_d} 1.
\]
We will use the following stationary phase lemma several times.

\begin{lemma}\label{lemma:stationary_phase}
  Suppose $w=w_T$ is a family of $X$-inert in $\xi$ with compact support on $[Z,2Z]$,
  so that $w^{(j)}(\xi) \ll (Z/X)^{-j}$.
  Suppose that on the support of $w$, $h=h_T$ is smooth and satisfies that
  $ h^{(j)}(\xi) \ll \frac{Y}{Z^j}$, for all $j\geq0.$
  Let
  \[
    I = \int_{\mathbb{R}} w(\xi) e^{i h(\xi)}  \dd \xi.
  \]
  \begin{enumerate}
  \item [(i)] If $h'(\xi) \gg \frac{Y}{Z}$ for all $\xi \in \supp w$. Suppose $Y/X \geq1$. Then
      $I \ll_A Z (Y/X)^{-A}$ for $A$ arbitrarily large.
  \item [(ii)] If $h''(\xi) \gg \frac{Y}{Z^2}$
  for all $\xi \in \supp w$, and there exists $\xi_0 \in\mathbb{R}$ such that $h'(\xi_0)=0$.
  Suppose that $Y/X^2 \geq 1$. Then
  \[
    I
    = \frac{e^{i h(\xi_0)}}{\sqrt{h''(\xi_0)}}  W_T(\xi_0) + O_A(Z(Y/X^2)^{-A}),
    \quad
    \textrm{for any $A>0$,}
  \]
  for some $X$-inert family of functions $W_T$ (depending on $A$) supported on $\xi_0\asymp Z$.
  \end{enumerate}
\end{lemma}


\begin{proof}
  See \cite[\S 8]{BlomerKhanYoung} and \cite[\S 3]{KPY}.
\end{proof}

\section{The dual sum and Proof of Theorem \ref{thm:1+3} }\label{sec:dual}

In this section, we reduce the estimate of $\mathcal A(X,1\boxplus\Phi)$ to its dual sum.
To avoid the use of the Ramanujan conjecture in the case $\Phi=\Sym^2\phi$, we don't use Perron's formula as Friedlander--Iwaniec \cite{FI2005} did. Instead, we use smoothing and Mellin transform to reduce the problem to an estimate of a first integral moment of L-functions. Then we apply the stationary phase method to deal with the first moment which reduces the problem to an estimate of its dual sum.

\subsection{Smoothing and Mellin transform}
Let $0<Y\leq X/5$. Let $W_1$ be a smooth function with support $\supp W_1 \in [1/2-Y/X,1+Y/X]$ such that $W_1(u) = 1$ if $u\in[1/2,1]$, $W_1(u)\in[0,1]$ if $u\in[1/2-Y/X,1/2]\cup [1,1+Y/X]$.
Similarly, let $W_2$ be a smooth function with support $\supp W_2 \in [1/2,1]$ such that $W_2(u) = 1$ if $u\in[1/2+Y/X,1-Y/X]$, $W_2(u)\in[0,1]$ if $u\in[1/2,1/2+Y/X]\cup [1-Y/X,1]$. Assume  $W_j^{(k)}(u) \ll (X/Y)^{k}$ for any integer $k\geq0$ and $j\in\{1,2\}$. Then by \eqref{eqn:nonnegativity}, we have
\begin{multline*}
  \sum_{n\geq1} \lambda_{1\boxplus\Sym^2\phi}(n) W_{2}\left(\frac{n}{X}\right) \leq \mathcal A(X,1\boxplus\Sym^2\phi)-\mathcal A(X/2,1\boxplus\Sym^2\phi)
   \\
  \leq \sum_{n\geq1} \lambda_{1\boxplus\Sym^2\phi}(n) W_{1}\left(\frac{n}{X}\right).
\end{multline*}
In order to prove Theorem \ref{thm:1+3}, it suffices to prove for $W\in\{W_1,W_2\}$ and
\[
  \textrm{$Y=X^{3/5-\delta}$, for some $\delta\in[0,1/10)$,}
\]
we have
\begin{equation}\label{eqn:smoothed}
  \sum_{n\geq1} \lambda_{1\boxplus\Sym^2\phi}(n) W\left(\frac{n}{X}\right) = L(1,\Sym^2\phi) \tilde{W}(1) X  +  O(X^{3/5-\delta+o(1)}),
\end{equation}
where $\tilde{W}(s)$ 
is the Mellin transform of $W$. Note that $\tilde{W}(1)=1/2+O(Y/X)$.
Indeed, \eqref{eqn:smoothed} leads to
\[
  \mathcal A(X,1\boxplus\Sym^2\phi)-\mathcal A(X/2,1\boxplus\Sym^2\phi) = L(1,\Sym^2\phi) \frac{X}{2}  +  O(X^{3/5-\delta+o(1)}).
\]
Hence we have $\mathcal A(X,1\boxplus\Sym^2\phi) = L(1,\Sym^2\phi) \, X  +  O(X^{3/5-\delta+o(1)})$.

Let $\Phi$ be a Hecke--Maass cusp form for $\SL(3,\mathbb Z)$. Assuming GRC for $\Phi$, we have
\[
  \lambda_{1\boxplus\Phi}(n) = \sum_{\ell m = n} A_\Phi(1,m) \ll n^{o(1)}.
\]
Hence by taking $Y=X^{3/5-\delta}$ we have
\begin{multline*}
  \mathcal A(X,1\boxplus\Phi)-\mathcal A(X/2,1\boxplus\Phi) - \sum_{n\geq1} \lambda_{1\boxplus\Phi}(n) W_{1}\left(\frac{n}{X}\right) \\
  \ll \sum_{X/2-Y<n<X} |\lambda_{1\boxplus\Phi}(n)| + \sum_{X/2<n<X+Y} |\lambda_{1\boxplus\Phi}(n)| \ll X^{3/5-\delta+o(1)}.
\end{multline*}
Hence it suffices to show for $W=W_1$ and $Y=X^{3/5-\delta}$, we have
\begin{equation*}
  \sum_{n\geq1} \lambda_{1\boxplus\Phi}(n) W\left(\frac{n}{X}\right) = L(1,\Phi) \tilde{W}(1) X  +  O(X^{3/5-\delta+o(1)}).
\end{equation*}
\begin{remark}
  This is the only place where we need to assume GRC for $\Phi$ in order to prove Theorem \ref{thm:1+3}. In fact it suffices to assume GRC on average in short intervals for $\Phi$.
\end{remark}

By the inverse Mellin transform, we have
\[
  W(u) = \frac{1}{2\pi i} \int_{(2)} \tilde{W}(s) u^{-s} \dd s.
\]
Hence we get
\[
  \sum_{n\geq1} \lambda_{1\boxplus\Phi}(n) W\left(\frac{n}{X}\right)
  = \frac{1}{2\pi i} \int_{(2)} \tilde{W}(s) L(s,1\boxplus\Phi) X^s \dd s.
\]
Since $L(s,1\boxplus\Phi)=\zeta(s)L(s,\Phi)$, we have $\Res_{s=1}L(s,1\boxplus\Phi)=L(1,\Phi)$.
Shifting the contour of integration to the left, we get
\[
  \sum_{n\geq1} \lambda_{1\boxplus\Phi}(n) W\left(\frac{n}{X}\right)
  = L(1,\Phi) \tilde{W}(1) X +  \frac{1}{2\pi i} \int_{(1/2)} \tilde{W}(s) L(s,1\boxplus\Phi) X^s \dd s.
\]
By repeated integration by parts, we have
\begin{equation}\label{eqn:W=}
  \tilde{W}(s) = -\frac{1}{s} \int_{0}^{\infty} W'(u) u^{s} \dd u \ll \frac{1}{|s|^k} \left(\frac{X}{Y}\right)^{k-1}, \quad \textrm{for any $k\geq1$},
\end{equation}
since $\supp W^{(k)} \in [1/2-Y/X,1/2+Y/X]\cup [1-Y/X,1+Y/X]$.
This allows us to truncate the $s$-integral at $|s|\ll X^{1+\varepsilon}/Y$. In fact, we apply a smooth partition of unity, getting
\begin{multline*}
  \sum_{n\geq1} \lambda_{1\boxplus\Phi}(n) W\left(\frac{n}{X}\right)
  = L(1,\Phi) \tilde{W}(1) X + O(X^{-2020}) \\
   + O\bigg(X^{1/2} \sum_{\substack{T\leq X^{1+\varepsilon}/Y \\ T \; \textrm{dyadic}}} \left| \int_{\mathbb{R}} \tilde{W}(1/2+it) V\left(\frac{t}{T}\right) L(1/2+it,1\boxplus\Phi) X^{it} \dd t \right| \bigg).
\end{multline*}
By the first equality in \eqref{eqn:W=}, we have
\begin{multline*}
  \int_{\mathbb{R}} \tilde{W}(1/2+it) V\left(\frac{t}{T}\right) L(1/2+it,1\boxplus\Phi) X^{it} \dd t
  \\ = -
  \int_{0}^{\infty} W'(u) u^{1/2} \int_{\mathbb{R}} V\left(\frac{t}{T}\right) L(1/2+it,1\boxplus\Phi) (uX)^{it} \frac{\dd t}{1/2+it} \dd u.
\end{multline*}
Recall that $\tilde{W}(1)=1/2+O(Y/X)$. Hence
\begin{multline}\label{eqn:sum2integral}
  \sum_{n\geq1} \lambda_{1\boxplus\Phi}(n) W\left(\frac{n}{X}\right)
  = L(1,\Phi) X/2 + O(Y) \\
   + O\bigg( \sup_{u\in [1/3,2]} \sup_{T\ll X^{1+\varepsilon}/Y } \frac{X^{1/2}}{T} \left| \int_{\mathbb{R}} V\left(\frac{t}{T}\right) L(1/2+it,1\boxplus\Phi) (uX)^{it} \dd t \right| \bigg),
\end{multline}
for some fixed $V$ with compact support. Hence it suffices to consider
\[
  \mathcal{I} := \int_{\mathbb{R}} V\left(\frac{t}{T}\right) L(1/2+it,1\boxplus\Phi) X^{it} \dd t.
\]
We only consider the case $T\geq1$, since the case $T\leq -1$ can be done similarly and the case $-1\leq T\leq 1$ can be treated trivially. We will prove the following proposition.
\begin{proposition}\label{prop:I}
  Let $\Phi$ be a Hecke--Maass cusp form for $\SL_3(\mathbb{Z})$.
  We have
  \begin{itemize}
    \item [(i)] For any $X>0$ and $T\geq1$, we have $\mathcal{I} \ll T^{5/4+\varepsilon}$.
    \item [(ii)] If $X^{\varepsilon} \leq T \leq X^{2/5}$, then we have
        $ \mathcal{I} \ll T^{5/2+\varepsilon}X^{-1/2} + T^{1+\varepsilon}$.
    \item [(iii)]
  If $X^{5/13} \leq T \leq X^{5/12}$, then we have
  \[
    \mathcal{I} \ll T^{56/25+\varepsilon} X^{-2/5} .
  \]
  \end{itemize}
\end{proposition}

We will prove this proposition in the next subsection which will need Theorem \ref{thm:3}. We can prove nontrivial results in other ranges of $T$ by using the second claim in Theorem \ref{thm:3}. Since this is enough for our application, we don't pursue it here.  Now we can finish the proof of Theorem \ref{thm:1+3}.

\begin{proof}[Proof of Theorem  \ref{thm:1+3} by assuming Proposition \ref{prop:I}]
  Assume $\delta<1/60$. Then by Proposition \ref{prop:I} (iii), for $X^{5/13} \leq T \leq X^{1+\varepsilon}/Y = X^{2/5+\delta+\varepsilon}$, the contribution to \eqref{eqn:sum2integral} is bounded by
  \[
    O(X^{1/2}T^{-1} T^{56/25+\varepsilon} X^{-2/5})
    = O(X^{1/10} T^{31/25+\varepsilon}) = O(X^{(149+310 \delta)/250+\varepsilon}).
  \]
  Note that the other error term is $O(X^{3/5-\delta+\varepsilon})$, so the best choice is $\delta=1/560$.
  By Proposition \ref{prop:I} (i), for $T \leq X^{5/13}$, the contribution to \eqref{eqn:sum2integral} is bounded by
  \[
    O(X^{1/2}T^{-1} T^{5/4+\varepsilon})
    = O(X^{1/2}T^{1/4+\varepsilon}) = O(X^{31/52+\varepsilon}) = O(X^{3/5-\delta+\varepsilon}).
  \]
  This proves Theorem  \ref{thm:1+3}.
\end{proof}

\subsection{The integral first moment}

We will treat $\mathcal{I}$ differently depending on the magnitudes of  $T$ and $X$.

For any  $X>0$, 
we can use the fact $L(1/2+it,1\boxplus\Phi)=\zeta(1/2+it)L(1/2+it,\Phi)$. Applying the Cauchy--Schwarz inequality,  we get
\[
  \mathcal{I}  \ll \Big( \int_{\mathbb{R}} V\left(\frac{t}{T}\right) |\zeta(1/2+it)|^2 \dd t \Big)^{1/2} \Big( \int_{\mathbb{R}} V\left(\frac{t}{T}\right) |L(1/2+it,\Phi)|^2 \dd t \Big)^{1/2}.
\]
By the approximate functional equations (see Lemma \ref{lemma:AFE}) for both $f=1$ and $f=\Phi$, we have
\[
  \mathcal{I}  \ll T^\varepsilon \Big( \int_{0}^{T} \Big|\sum_{\ell\leq T^{1/2+\varepsilon}} a_\ell \ell^{it}\Big|^2 \dd t \Big)^{1/2} \Big( \int_{0}^{T} \Big|\sum_{m\leq T^{3/2+\varepsilon}} b_m m^{it}\Big|^2 \dd t \Big)^{1/2},
\]
for some $\{a_\ell\}$ and $\{b_m\}$ satisfying that $|a_\ell|\ll \ell^{-1/2+\varepsilon}$ and $\sum_{m\leq M} |b_m|^2 \ll M^{\varepsilon}$.
Here we have used \eqref{eqn:RS3}.
Now by the integral mean-value estimate (see e.g. \cite[Theorem 9.1]{IwaniecKowalski2004analytic})
\begin{equation}\label{eqn:IMV}
  \int_{0}^{T} \Big| \sum_{n\leq N} a_n n^{it} \Big|^2 \dd t \ll (T+N) \sum_{n\leq N} |a_n|^2,
\end{equation}
we obtain
\begin{equation}\label{eqn:I<<}
  \mathcal{I} \ll T^{5/4+\varepsilon}.
\end{equation}
This proves Proposition \ref{prop:I} (i).

If we assume $T\leq X^{2/5-4\delta+\varepsilon}$, we find the contribution from the error term in \eqref{eqn:sum2integral} is
\[
  O(X^{1/2+\varepsilon} T^{1/4}) = O(X^{3/5-\delta+o(1)}).
\]
\begin{remark}
  If we take $\delta=0$, then $X^{1+\varepsilon}/Y \leq X^{2/5+\varepsilon}$. Hence we repove the classical Rankin--Selberg estimate $\Delta_2(X,\phi) \ll X^{3/5+\varepsilon}$.
\end{remark}

From now on, we assume $X^\varepsilon  \leq T \leq X^{1/2-\varepsilon}$.
By the approximate functional equations (see Lemma \ref{lemma:AFE}) with $f=1\boxplus\Phi$ and $\Phi$ a Hecke--Maass cusp form for $\SL_3(\mathbb{Z})$, we have
\begin{multline*}
  \mathcal{I} \ll  T^\varepsilon \sup_{w\in[\varepsilon-iT^\varepsilon,\varepsilon+iT^\varepsilon]}
  \sup_{N\leq T^{2+\varepsilon}}
   \bigg(\bigg|
  \int_{\mathbb{R}} V\left(\frac{t}{T}\right) \sum_{n\geq1} \frac{\lambda_{1\boxplus\Phi}(n)}{n^{1/2+it+w}} V\left(\frac{n}{N}\right)  t^{2w} X^{it} \dd t\bigg|
  \\
  +  \bigg|
  \int_{\mathbb{R}} V\left(\frac{t}{T}\right) \sum_{n\geq1}  \frac{\overline{\lambda_{1\boxplus\Phi}(n)}}{n^{1/2-it+w}}  V\left(\frac{n}{N}\right)  \left(\frac{t}{2\pi e}\right)^{-i4t} t^{2w} X^{it} \dd t\bigg| \bigg)
  + O\left(T^{1+\varepsilon}\right).
\end{multline*}
We can absorb the factor $t^{2w}$ to the weight function $V(t/T)$ and $1/n^{1/2+w}$ to $V(n/N)$. Then the new $V_j$ ($j=1,2$) depends on $w$ and  satisfies that $\supp V_j\subset (1/2,1)$ and $V_j^{(k)} \ll T^{k\varepsilon}$ for $k\geq0$. Hence we have
\begin{equation*}
  \mathcal{I} \ll  T^\varepsilon N^{-1/2} \sup_{w\in[\varepsilon-iT^\varepsilon,\varepsilon+iT^\varepsilon]}
  \sup_{N\leq T^{2+\varepsilon}}  \big(\big| \mathcal{I}_1(N) \big|
  +  \big| \mathcal{I}_2(N) \big| \big)
  + O\left(T^{1+\varepsilon}\right),
\end{equation*}
where
\[
  \mathcal{I}_1(N) := \int_{\mathbb{R}} V_1\left(\frac{t}{T}\right) \sum_{n\geq1} \frac{\lambda_{1\boxplus\Phi}(n)}{n^{it}} V_2\left(\frac{n}{N}\right)  X^{it} \dd t
\]
and
\[
  \mathcal{I}_2(N) := \int_{\mathbb{R}} \overline{V_1}\left(\frac{t}{T}\right) \sum_{n\geq1}  \frac{\lambda_{1\boxplus\Phi}(n)}{n^{it}}  \overline{ V_2}\left(\frac{n}{N}\right)  \left(\frac{t}{2\pi e}\right)^{i4t} X^{-it} \dd t,
\]
where $\overline{V_j}(u)=\overline{V_j(u)}$ for $j=1,2$.

We first deal with $\mathcal{I}_1(N)$. Changing the order of integral and summation, and making a change of variable $t=T\xi$, we get
\[
  \mathcal{I}_1(N) = T \sum_{n\geq1} \lambda_{1\boxplus\Phi}(n) V_2\left(\frac{n}{N}\right) \int_{\mathbb{R}} V_1\left(\xi\right) e^{i\xi T\log X/n} \dd \xi.
\]
Since $n\ll N \ll T^{2+\varepsilon}  \ll X^{1-\varepsilon}$, we have $T\log X/n \gg T$. By repeated integration by parts, we obtain
\begin{equation}\label{eqn:I1<<}
  \mathcal{I}_1(N) = O(T^{-2020}).
\end{equation}
Now we consider $\mathcal{I}_2(N)$. Similarly, we arrive at
\[
  \mathcal{I}_2(N) = T \sum_{n\geq1} \lambda_{1\boxplus\Phi}(n)  \overline{V_2}\left(\frac{n}{N}\right) \int_{\mathbb{R}} \overline{V_1}\left(\xi\right) e^{i (4 \xi T \log \xi + \xi T \log \frac{T^4}{(2\pi e)^4 nX})}   \dd \xi.
\]
Let
\[
  h_1(\xi) = 4 \xi T \log \xi + \xi T \log \frac{T^4}{(2\pi e)^4 nX}.
\]
Then we have
\[
  h_1'(\xi) = 4 T \log e \xi + T \log \frac{T^4}{(2\pi e)^4 nX}  = 4 T \log \frac{T \xi}{2\pi (nX)^{1/4}},
\]
\[
  h_1''(\xi) = 4T /\xi, \quad h_1^{(j)}(\xi) \asymp_j T, \quad j\geq 2.
\]
The solution of $h_1'(\xi)=0$ is $\xi_0=\frac{2\pi (nX)^{1/4}}{T}$.
Note that $h_1(\xi_0)= -8\pi (nX)^{1/4}$.
By the stationary phase method for the $\xi$-integral (see Lemma \ref{lemma:stationary_phase} (ii)), we get
\[
  \mathcal{I}_2(N) = T^{1/2} \sum_{n\geq1}  \lambda_{1\boxplus\Phi}(n)  \overline{V_2}\left(\frac{n}{N}\right) e\left(-4(nX)^{1/4}\right) V_3\left(\frac{(nX)^{1/4}}{T}\right) + O(T^{-2020}),
\]
where $V_3$ is some smooth function such that $\supp V_3\subset (1/20,20)$ and $V_3^{(k)} \ll T^{k\varepsilon}$ for $k\geq0$.
Hence we only need to consider $N\asymp T^4/X$, otherwise the contribution is negligibly small.
Thus if $T\leq X^{1/4-\varepsilon}$, then we have
\begin{equation}\label{eqn:I2<<s}
  \mathcal{I}_2(N) = O(T^{-2020}).
\end{equation}

Now we assume $T\geq X^{1/4-\varepsilon}$.
When $N\asymp T^4/X$, we can remove the weight function $V_3$ by a Mellin inversion, getting
\[
  \mathcal{I}_2(N) = T^{1/2} \frac{1}{2\pi}\int_{\mathbb{R}} \sum_{n\geq1}  \lambda_{1\boxplus\Phi}(n)  \overline{V_2}\left(\frac{n}{N}\right) e\left(-4(nX)^{1/4}\right) \tilde{V_3}(iv)  \left(\frac{(nX)^{1/4}}{T}\right)^{-iv} \dd v + O(T^{-2020}).
\]
By repeated integration by parts, we can truncate $v$-integral at $|v|\leq T^\varepsilon$ with a negligible error term.
Hence we have
\begin{equation}\label{eqn:I2<<}
  \mathcal{I}_2(N) \ll  T^{1/2+\varepsilon} \sup_{|v|\leq T^\varepsilon} \Big| \sum_{n\geq1}  \lambda_{1\boxplus\Phi}(n) V\left(\frac{n}{N}\right) e\left(-4(nX)^{1/4}\right) \Big| + T^{-2020},
\end{equation}
for some smooth function $V$ (depending on $v$) such that $\supp V\subset (1/20,20)$ and $V^{(k)} \ll T^{k\varepsilon}$ for $k\geq0$.
Thus by \eqref{eqn:I1<<} and \eqref{eqn:I2<<}, we get
\begin{equation}\label{eqn:I<<sum}
  \mathcal{I} \ll  T^\varepsilon \sup_{w\in[\varepsilon-iT^\varepsilon,\varepsilon+iT^\varepsilon]}
  \sup_{N\asymp T^{4}/X} \sup_{|v|\leq T^\varepsilon}
   \frac{T^{1/2}}{ N^{1/2}}
  \Big| \sum_{n\geq1}  \lambda_{1\boxplus\Phi}(n) V\left(\frac{n}{N}\right) e\left(-4(nX)^{1/4}\right) \Big|
  + O\left(T^{1+\varepsilon}\right).
\end{equation}
By \eqref{eqn:RS3}, we get
\begin{equation}\label{eqn:I2<<m}
  \mathcal{I} \ll T^{5/2+\varepsilon}X^{-1/2} + T^{1+\varepsilon}.
\end{equation}
This proves Proposition \ref{prop:I} (ii).

To prove Proposition \ref{prop:I} (iii), we need to find a power saving in the dual sum
\[
  \mathcal B(N) := \sum_{n\geq1}  \lambda_{1\boxplus\Phi}(n) V\left(\frac{n}{N}\right) e\left(-4(nX)^{1/4}\right)
\]
if $X^{5/13} \leq T \leq X^{5/12}$.
Now we should use the fact $\lambda_{1\boxplus\Phi}(n) = \sum_{\ell m=n}A(1,m)$. Applying a dyadic partition of the $\ell$-sum and a smooth partition of unity for the $m$-sum, we get
\[
  \mathcal B(N) = \sum_{\substack{L\ll N \\ L \; \textrm{dyadic}}} \sum_{\substack{M\ll N \\ M \; \textrm{dyadic}}} \sum_{L<\ell\leq 2L} \sum_{m\geq1}  A(1,m) W\left(\frac{m}{M}\right) V\left(\frac{\ell m}{N}\right) e\left(- 4(\ell mX)^{1/4}\right).
\]
Here  $W$ is a fixed smooth function such that $\supp W\subset [1/4,2]$ and $W^{(k)} \ll 1$ for $k\geq0$.
Because of $\supp V\subset (1/20,20)$, we only need to consider the case $LM \asymp N$, in which case we can remove the weight function $V$ by a Mellin inversion as above.
Hence we obtain
\begin{equation}\label{eqn:B(N)<<B(LM)}
  \mathcal{B}(N) \ll T^\varepsilon \sup_{\substack{L\ll N, \ M\ll N \\ LM\asymp N}} \sup_{-T^\varepsilon \leq   v' \leq T^\varepsilon} |\mathcal{B}(L,M;v')|,
\end{equation}
where
\[
  \mathcal B(L,M;v') := \sum_{L<\ell\leq 2L}  \sum_{m\geq1} A(1,m) W\left(\frac{m}{M}\right)  e\left(- 4(\ell mX)^{1/4}\right) \ell^{iv'} m^{iv'}.
\]

We have the following estimates.
\begin{proposition}\label{prop:B}
  Let $X^{5/13} \leq T \leq X^{5/12}$, $N\asymp T^4/X \asymp LM$. Assume $-T^\varepsilon \leq   v' \leq T^\varepsilon$. Then we have
  \[
    \mathcal B(L,M;v') \ll T^{187/50+\varepsilon} X^{-9/10}.
  \]
\end{proposition}

In order to prove Proposition \ref{prop:B}, we need the following van der Corput type estimate of exponential sums.
\begin{lemma}\label{Lemma:FI}
  Let $h$ be a smooth function on the interval $[L,2L]$ with derivatives satisfying that $|h^{(k)}|\asymp FL^{-k}$ for $1\leq k\leq 5$.  Then for any subinterval $I$ of $[L,2L]$ we have
  \[
    \sum_{\ell\in I} e(h(\ell)) \ll F^{1/30} L^{5/6} + F^{-1} L.
  \]
\end{lemma}

\begin{proof}
  This is Theorem 2.9 in \cite{GK} with $q=3$.
\end{proof}

\begin{proof}[Proof of Proposition \ref{prop:B} by assuming Theorem \ref{thm:3}]
  If $L\gg T^{44/25}X^{-3/5}$, then by Lemma \ref{Lemma:FI} with $F=T$, the Cauchy--Schwarz inequality for the $m$-sum, and \eqref{eqn:RS3}, we have
  \begin{align*}
    \mathcal B(L,M;v') & \leq  \sum_{m\asymp M}  |A(1,m)|  \Big| \sum_{L<\ell\leq 2L}  e\left(- 4(\ell mX)^{1/4} + \frac{v'}{2\pi} \log \ell \right) \Big| \\
    & \ll T^\varepsilon M (T^{1/30} L^{5/6} + T^{-1} L) \\
    & \ll T^\varepsilon \frac{T^4}{X} (T^{1/30} L^{-1/6} + T^{-1})
    \ll T^{187/50+\varepsilon} X^{-9/10}.
  \end{align*}
  If $L\ll T^{44/25}X^{-3/5}$, then
  \[
    M\asymp N/L \gg (T^{4}/X)(T^{44/25}X^{-3/5})^{-1} \asymp T^{56/25}/X^{2/5} \geq T^{6/5}
  \]
  provided $T\geq X^{5/13}$. Note that $M\leq N \ll T^{4}/X \ll T^{8/5}$ provided $T\leq X^{5/12}$.
  By the first claim in Theorem \ref{thm:3} we have
  \begin{align*}
    \mathcal B(L,M;v') & \leq  \sum_{L<\ell\leq 2L} \Big| \sum_{m\geq1}  A(1,m) m^{iv'} W\left(\frac{m}{M}\right)  e\left(- 4(\ell mX)^{1/4}\right) \Big| \\
    & \ll L T^{3/10+\varepsilon} M^{3/4} \ll T^{3/10+\varepsilon}L^{1/4} N^{3/4}
    \ll T^{187/50+\varepsilon} X^{-9/10}.
  \end{align*}
  This completes the proof of Proposition \ref{prop:B}.
\end{proof}

\begin{proof}[Proof of Proposition \ref{prop:I} (iii)]
  For $X^{5/13} \leq T \leq X^{5/12}$, by \eqref{eqn:I<<sum}, \eqref{eqn:B(N)<<B(LM)}, and Proposition \ref{prop:B}, we have
  \[
    \mathcal I \ll T^\varepsilon
    \sup_{N\asymp T^{4}/X}
    \frac{T^{1/2}}{ N^{1/2}}
    T^{187/50+\varepsilon} X^{-9/10}
    + T^{1+\varepsilon}
    \ll T^{56/25+\varepsilon} X^{-2/5}
  \]
  as claimed.
\end{proof}

\section{Proof of Theorem \ref{thm:RS}}\label{sec:proofRS}

The proof is standard once we have Theorem \ref{thm:1+3}. For completeness, we include the proof here.
Let $\phi$ be a $\GL(2)$ Hecke--Maass cusp form for $\SL(2,\mathbb{Z})$ with its $n$-th Hecke eigenvalue $\lambda_\phi(n)$.
Note that $\phi\times \phi = 1\boxplus \Sym^2 \phi$.
By Theorem \ref{thm:1+3} with $\Phi=\Sym^2 \phi$,  we have
\begin{equation}\label{eqn:SumRS}
  \sum_{n\leq X} \lambda_{\phi\times\phi}(n) = L(1,\Sym^2 \phi) X + O(X^{3/5-\delta+o(1)}).
\end{equation}
Moreover, by \eqref{eqn:RS=1+Sym2}, we have
\[
  \lambda_\phi(n)^2 = \mathop{\sum\sum}_{\ell^2m=n} \mu(\ell) \lambda_{\phi\times\phi}(m).
\]
Hence by \eqref{eqn:SumRS} we have
\begin{align*}
  \sum_{n\leq X} \lambda_\phi(n)^2
  & = \mathop{\sum\sum}_{\ell^2m\leq X} \mu(\ell) \lambda_{\phi\times\phi}(m) \\
  & = \sum_{\ell\leq X^{1/2}} \mu(\ell)  \left( L(1,\Sym^2 \phi) \frac{X}{\ell^2} + O(X^{3/5-\delta+o(1)}\ell^{-6/5+2\delta}) \right) \\
  & = \frac{L(1,\Sym^2 \phi)}{\zeta(2)} X + O(X^{3/5-\delta+o(1)}),
\end{align*}
provided $\delta<1/10$. This completes the proof of Theorem \ref{thm:RS}.

\section{Proof of Theorem \ref{thm:3}}\label{sec:proof3}

In this section we prove Theorem \ref{thm:3}. We will not use the exact expression of $\varphi$ until the end of the proof.

\subsection{Applying the delta method}\label{subsec:delta}

We apply \eqref{eqn:delta-n} to $\mathscr S(N)$ as a device to separate the variables.
By \eqref{eqn:delta-n} with some large $Q$,  we get that $\mathscr S(N)$ is equal to
\begin{multline*}
    \sum_{m\geq1} \sum_{n\geq1} A(1,n) e\left(T \varphi\left(\frac{m}{N}\right)\right)V\Big(\frac{m}{N}\Big) W\Big(\frac{n}{N}\Big) \delta(m-n)
  \\
  =
   \sum_{m\geq1} \sum_{n\geq1} A(1,n) e\left(t \varphi\left(\frac{m}{N}\right)\right)V\Big(\frac{m}{N}\Big) W\Big(\frac{n}{N}\Big) \\
 \cdot \frac{1}{Q}\sum_{1\leq q\leq Q} \;\frac{1}{q}\; \sideset{}{^\star}\sum_{a\bmod{q}}e\left(\frac{(m-n)a}{q}\right) \int_\mathbb{R}g(q,x) e\left(\frac{(m-n) x}{qQ}\right)\mathrm{d}x .
\end{multline*}
Here $W$ is a fixed smooth function such that $W(x)=1$ if $x\in[1/2,1]$, $\supp W\subset [1/4,2]$, and $W^{(k)}\ll 1$.
By \eqref{eqn:RS3} and \eqref{eqn:g-h}, we have
\begin{multline*}
  \mathscr S(N)
   = \int_{\mathbb{R}} U\left(\frac{x}{T^\varepsilon}\right) \frac{1}{Q}\sum_{1\leq q\leq Q} \;\frac{g(q,x)}{q}\; \sideset{}{^\star}\sum_{a\bmod{q}} \sum_{m\geq1} e\left(T \varphi\left(\frac{m}{N}\right)\right) e\left(\frac{ma}{q}\right) V\Big(\frac{m}{N}\Big) e\left(\frac{m x}{qQ}\right)
  \\
 \cdot \sum_{n\geq1} A(1,n) e\left(\frac{-na}{q}\right) W\Big(\frac{n}{N}\Big) e\left(\frac{-nx}{qQ}\right)
 \dd x + O(T^{-A}),
\end{multline*}
where $U$ is a smooth positive function with $U(x)=1$ if $x\in[-1,1]$, supported in $[-2,2]$ and satisfying $U^{(j)}(x)\ll_j 1$.

\subsection{Applications of summation formulas} \label{subsec:summation_formula}

Now we apply the Poisson summation formula to the $m$-sum in $\mathscr S(N)$, getting
\begin{align*}
  \textrm{$m$-sum}
  & = \sum_{b \bmod{q}}e\left(\frac{ab}{q}\right)  \sum_{m\equiv b\bmod{q}} e\left(T \varphi\left(\frac{m}{N}\right)\right)
   V\left(\frac{m}{N}\right) e\left(\frac{m x}{qQ}\right) \\
  & = \frac{N}{q} \sum_{m\in\mathbb{Z}} \sum_{b \bmod{q}} e\left(\frac{(m+a)b}{q}\right)  \int_{\mathbb{R}} V\left(y\right) e\left(T\varphi(y)+\frac{Nxy}{qQ}\right) e\left(-\frac{mN}{q}y\right) \dd y \\
  & = N \sum_{m \equiv a \bmod{q}} \mathcal V (m,q,x),
\end{align*}
where
\[
   \mathcal V (m,q,x) := \int_{\mathbb{R}} V\left(y\right) e\left(T\varphi(y)+\frac{Nxy}{qQ}+\frac{mN}{q}y\right) \dd y.
\]
Assume $Q \geq NT^{2\varepsilon-1}$. Then for $x\ll T^\varepsilon$, we have $\frac{Nx}{qQ} \ll T^{1-\varepsilon}$.
 Since $\varphi'(y)\varphi''(y) \neq 0$ if $y\in(1/4,2)$, by repeated integration by parts we know that we can truncate $m$-sum at $|m| \asymp q T /N$, in which case
 we have $(T\varphi(y)+\frac{Nxy}{qQ}+\frac{mN}{q}y)'' = T\varphi''(y)\gg T$, and hence
\begin{equation}\label{eqn:V<<}
  \mathcal V (m,q,x) \ll T^{-1/2},
\end{equation}
by the second derivative test (see e.g. Huxley \cite[Lemma 5.1.3]{Huxley}).
Hence we can restrict $q$-sum with $N/T \ll q \leq Q$.

Now we consider the $n$-sum. Note that we have $q\gg N/T$.
Let $\psi_x(n)=W\left(\frac{n}{N}\right) e\left(\frac{-n x}{qQ}\right)$.
By Lemma \ref{lemma:VSF}, we have
\begin{equation*}
  \textrm{$n$-sum} = \frac{q\pi^{3/2}}{2} \sum_{\pm} \sum_{n_1|q} \sum_{n_2=1}^{\infty}
  \frac{A(n_2,n_1)}{n_1n_2}
  S\left(-\bar{m},\pm n_2;\frac{q}{n_1}\right)
  \Psi_x^{\pm}\left(\frac{n_1^2n_2}{q^3}\right),
\end{equation*}
where $\Psi_x^{\pm}$ is defined as in \eqref{eqn:Psi} with $\psi=\psi_x$.
Hence we have
\begin{multline*}
  \mathscr S(N)
   = \int_{\mathbb{R}} U\left(\frac{x}{T^\varepsilon}\right) \frac{N}{Q}\sum_{1\leq q\leq Q} g(q,x) \sum_{\substack{|m| \asymp q T /N\\(m,q)=1}} \mathcal V (m,q,x)
  \\
 \cdot \frac{\pi^{3/2}}{2} \sum_{\pm} \sum_{n_1|q} \sum_{n_2=1}^{\infty}
  \frac{A(n_2,n_1)}{n_1n_2}
  S\left(-\bar{m},\pm n_2;\frac{q}{n_1}\right)
  \Psi_x^{\pm}\left(\frac{n_1^2n_2}{q^3}\right)
 \dd x + O(T^{-A}).
\end{multline*}

\subsection{Analysis of the integrals} \label{subsec:integrals}

In this subsection we want to consider $\Psi^\pm_x(z)$. We prove the following lemma.
\begin{lemma}\label{lemma:Psi}
  Let $Y\in\mathbb{R}$ and $N\geq1$. Let $T\geq1$ be sufficiently large. Let $\psi(n)=W\left(n/N\right)e\left(-Yn/N\right)$, where $W$ is a fixed smooth function, compactly supported on $[1,2]$. Define $\Psi^{\pm}$ as in \eqref{eqn:Psi}. Then we have
  \begin{itemize}
    \item [(i)] If $zN\gg T^{\varepsilon}$, then $\Psi^\pm$ is negligibly small unless $\sgn(Y) = \pm$ and $\pm Y \asymp (zN)^{1/3}$, in which case we have
        \begin{equation}\label{eqn:Psi=1}
            \Psi^\pm(z) = e\left(\pm 2 \frac{(zN)^{1/2}}{(\pm Y)^{1/2}}\right) (zN)^{1/2} w\left(\frac{(zN)^{1/2}}{(\pm Y)^{3/2}}\right) + O(T^{-A}) \ll (zN)^{1/2},
        \end{equation}
        where $w$ is a certain compactly supported $1$-inert  function depending on $A$.
    \item [(ii)] If $zN\ll T^{\varepsilon}$ and $Y\gg T^\varepsilon$, then $\Psi^\pm(z) \ll_A T^{-A}$ for any $A>0$.
    \item [(iii)] If $zN\ll T^{\varepsilon}$ and $Y\ll T^\varepsilon$, then $\Psi^\pm(z) \ll T^{\varepsilon}$.
  \end{itemize}
\end{lemma}

\begin{proof}
If $zN\gg T^{\varepsilon}$, then  by Lemma \ref{lemma:Psi=M+O}  we have
\begin{align*}
  \Psi^\pm(z) & = z \int_0^\infty \psi(\xi) \sum_{\ell=1}^{L} \frac{\gamma_\ell}{(z \xi)^{\ell/3}}
  e\left(\pm3(z \xi)^{1/3}\right) \dd \xi
  + O\left((z N)^{1-L/3}\right) \\
  & = (zN)^{2/3} \int_0^\infty
  W\left(\xi\right) \sum_{\ell=1}^{L} \frac{\gamma_\ell (z N)^{(1-\ell)/3}}{\xi^{\ell/3}}
  e\left(-Y \xi \pm3(z N \xi)^{1/3}\right) \dd \xi
  + O\left((z N)^{1-L/3}\right).
\end{align*}
Let
\[
  h_2(\xi) = - 2\pi Y \xi \pm  6\pi (z N \xi)^{1/3}.
\]
Then we have
\[
  h_2'(\xi) = - 2\pi Y  \pm  2\pi (z N)^{1/3}  \xi^{-2/3},
\]
\[
  h_2''(\xi) = \mp \frac{4}{3} \pi (z N)^{1/3}  \xi^{-5/3}, \quad
  h_2^{(j)}(\xi) \asymp_j (z N)^{1/3} , \quad j\geq 2.
\]
For $\xi\asymp 1$, if $\pm \sgn(Y) = -1$, then $h_2'(\xi)\gg (zN)^{1/3}\gg T^\varepsilon$. By repeated integration by parts we obtain $\Psi^\pm \ll T^{-A}$ for any $A>0$. Assume $\sgn(Y)=\pm$. The solution of $h_2'(\xi)=0$ is $\xi_0=\frac{(zN)^{1/2}}{(\pm Y)^{3/2}}$.
Note that $h_2(\xi_0)=\pm 4\pi \frac{(zN)^{1/2}}{(\pm Y)^{1/2}}$.
By the stationary phase method (Lemma \ref{lemma:stationary_phase}), we prove Lemma \ref{lemma:Psi} (i).


Now we consider the case  $zN\ll T^\varepsilon$ and $Y \gg T^{\varepsilon}$. Note that
\begin{align*}
  \Psi^\pm(z) & = z \frac{1}{2\pi i} \int_{(1/2)} (\pi^3 z)^{-s} \gamma^\pm(s)  \int_{0}^{\infty} W\left(\frac{u}{N}\right) e\left(\frac{-u Y}{N}\right) u^{-s} \dd u  \dd s \\
  & = (zN)^{1/2} \frac{1}{2\pi^{5/2} } \int_{\mathbb{R}} (\pi^3 zN)^{-i\tau} \gamma^\pm(1/2+i\tau) \int_{0}^{\infty} W\left(\xi\right) e\left(-Y \xi \right) \xi^{-1/2-i\tau} \dd \xi  \dd \tau.
\end{align*}
We first consider the $\xi$-integral above:
\[
  \int_{0}^{\infty} W\left(\xi\right) e\left(-Y \xi \right) \xi^{-1/2-i\tau} \dd \xi
  =
  \int_{0}^{\infty} W\left(\xi\right)\xi^{-1/2}  e^{i(-2\pi Y \xi -\tau \log \xi)} \dd \xi.
\]
Let
\[
  h_3(\xi) = -2\pi Y \xi -\tau \log \xi.
\]
Then we have
\[
  h_3'(\xi) = -2\pi Y  -\tau/\xi, \quad
  h_3''(\xi) = \tau/\xi^2, \quad
  h_3^{(j)}(\xi) \asymp_j |\tau|, \quad j\geq 2.
\]
By repeated integration by parts we obtain $\Psi^\pm \ll T^{-A}$ for any $A>0$ unless $\sgn(\tau)=-\sgn(Y)$ and $|\tau|\asymp |Y|$ which we assume now.
The solution of $h_3'(\xi)=0$ is $\xi_0=\frac{-\tau}{2\pi Y}$.
Note that $h_3(\xi_0)=\tau-\tau \log \frac{-\tau}{2\pi Y}$.
By Lemma \ref{lemma:stationary_phase} (ii) for the $\xi$-integral, we obtain
\[
  \textrm{$\xi$-integral} = \frac{e^{i\tau - i\tau \log \frac{-\tau }{2\pi Y}}} {\sqrt{Y}} w_1\left(\frac{-\tau }{Y}\right) + O(T^{-A}),
\]
where $w_1$ is a compactly supported $1$-inert function.
Since $\tau \gg T^{\varepsilon}$, by Stirling's formula, we have
\[
  \gamma^\pm(1/2+i\tau) = \left(\frac{|\tau|}{2e}\right)^{3i\tau} \Upsilon^\pm(\tau) + O(T^{-A}),
\]
where $(\Upsilon^\pm)^{(k)}(\tau)\ll |\tau|^{-k}$ for $k\geq0$.
Hence to bound $\Psi^\pm(z)$, we need to consider (making a change of variable $\tau\rightsquigarrow -\tau$)
\[
  \int_{\mathbb{R}} w_1\left(\frac{\tau}{Y}\right) \Upsilon^\pm(-\tau) e^{-i\tau + i\tau \log \frac{\tau}{2\pi Y}+i\tau \log (\pi^3 zN) - 3i\tau \log \frac{\tau}{2e} } \dd \tau.
\]
Let
\[
  h_4(\tau) = -\tau + \tau \log \frac{\tau}{2\pi Y}+ \tau \log (\pi^3 zN) - 3 \tau \log \frac{\tau}{2e}.
\]
Then we have
\[
  h_4'(\tau) = \log \frac{\tau}{2\pi Y}
  + \log (\pi^3 zN) - 3 \log \frac{\tau}{2e} - 2
  = \log \frac{4 \pi^2 e  zN }{Y\tau^2} \gg \varepsilon \log T,
\]
\[
  h_4''(\tau) = -2/\tau, \quad   h_4^{(j)}(\tau) \ll_j  |\tau|^{1-j}, \quad j\geq2.
\]
Hence $\Psi^\pm(z)$ is negligibly small by Lemma \ref{lemma:stationary_phase} (i).

Finally we handle the case $zN\ll T^\varepsilon$ and  $Y \ll T^{\varepsilon}$. By the first derivative test for the $\xi$-integral, we know that
$\Psi^\pm(z)$ is negligible unless $|\tau|\ll T^\varepsilon$. Hence $\Psi^\pm(z) \ll T^\varepsilon$.
\end{proof}

We now break the $q$-sum into dyadic segments $R< q \leq 2R$ with $N/T\ll R \ll Q$ and insert a smooth partition of unity for the $x$-integral and absorb the weight function $U(x/T^\varepsilon)$, getting
\[
  \mathscr S(N) \ll N^\varepsilon \sup_{T^{-100}\leq X \leq T^\varepsilon}
     \sup_{N/T\ll R \ll Q}  |\mathscr S^\pm(N;X,R)| + T^{-20},
\]
where
\begin{multline*}
  \mathscr S^\pm(N;X,R) :=
   \int_{\mathbb{R}} W\left(\frac{\pm x}{X}\right)\frac{N}{Q}\sum_{q\sim R} g(q,x)
   \sum_{\substack{|m| \asymp q T /N\\(m,q)=1}} \mathcal V (m,q,x)
  \\
 \cdot  \sum_{\pm} \sum_{n_1|q} \sum_{n_2=1}^{\infty}
  \frac{A(n_2,n_1)}{n_1n_2}
  S\left(-\bar{m},\pm n_2;\frac{q}{n_1}\right)
  \Psi_x^{\pm}\left(\frac{n_1^2n_2}{q^3}\right)
 \dd x,
\end{multline*}
for some compactly supported $1$-inert function $W$.

If $R\leq Q T^{-\varepsilon}$ and $X\leq T^{-\varepsilon}$, then by \eqref{eqn:g-h} we can replace $g(q,x)$ by 1 with a negligible error term.
Hence we obtain
\begin{align}\label{eqn:Sneq_error1}
  \mathscr S^\pm(N;X,R) = \mathscr S_1^\pm(N;X,R)
  + O(T^{-A})  ,
\end{align}
where
\begin{multline}\label{eqn:S1}
  \mathscr S_1^\pm(N;X,R) :=
   \frac{N}{Q}\sum_{q\sim R} \int_{\mathbb{R}} W_q\left(\frac{\pm x}{X}\right) \sum_{\substack{m\asymp qT/N \\ (m,q)=1}} \mathcal V (m,q,x)
  \\
 \cdot \sum_{\sigma\in\{\pm\}} \sum_{n_1|q} \sum_{n_2=1}^{\infty}
  \frac{A(n_2,n_1)}{n_1n_2}
  S\left(-\bar{m},\sigma n_2;\frac{q}{n_1}\right)
  \Psi_x^{\sigma}\left(\frac{n_1^2n_2}{q^3}\right)
 \dd x,
\end{multline}
with $W_q\left(\frac{\pm x}{X}\right)=W\left(\frac{\pm x}{X}\right)$ if $R\leq Q T^{-\varepsilon}$ and $X\leq T^{-\varepsilon}$, and
$W_q\left(\frac{\pm x}{X}\right)=W\left(\frac{\pm x}{X}\right)g(q,x)$ otherwise.

We first assume that  $NX/RQ\gg T^{2\varepsilon}$.
If $n_1^2n_2N/q^3 \gg T^\varepsilon$, then by \eqref{eqn:Psi=1} we have
\[
  \Psi_x^{\sigma}\left(\frac{n_1^2n_2}{q^3}\right)
  = e\left(2 \sigma \frac{(n_1^2n_2 Q)^{1/2}}{q(\sigma x)^{1/2}}\right) \left(\frac{n_1^2n_2 N}{q^3}\right)^{1/2} w_{2}\left(\frac{(n_1^2n_2 N)^{1/2}}{(\sigma Nx/Q)^{3/2}}\right) + O(T^{-A}),
\]
for some compactly supported $1$-inert  function $w_2$.
Hence the contribution to $\mathscr S_1^\pm(N;X,R)$ is negligible unless $\sigma = \sgn (x)$.
Thus in this case, up to a negligible error term, the contribution to $\mathscr S_1^\pm(N;X,R)$ is equal to
\begin{multline*}
   \frac{N^{3/2}}{Q}
   \sum_{q\sim R} \int_{\mathbb{R}} W_q\left(\frac{\pm x}{X}\right) \;\frac{1}{q^{3/2}}\; \sum_{\substack{m\asymp qT/N \\ (m,q)=1}} \mathcal V (m,q,x)
  \\
 \cdot \sum_{n_1|q} \sum_{n_2=1}^{\infty}
  \frac{A(n_2,n_1)}{n_2^{1/2}}
  S\left(-\bar{m},\pm n_2;\frac{q}{n_1}\right)
  e\left(\pm 2\frac{(n_1^2n_2 Q)^{1/2}}{q(\pm x)^{1/2}}\right) w_{2}\left(\frac{(n_1^2n_2 N)^{1/2}}{(\pm Nx/Q)^{3/2}}\right)
 \dd x.
\end{multline*}
Making a change of variable $\pm x / X \mapsto v$, we arrive at
\begin{multline*}
   \pm \frac{N^{3/2} X}{Q}
   \sum_{q\sim R} \int_{\mathbb{R}} W_q\left(v\right) \;\frac{1}{q^{3/2}}\; \sum_{\substack{m\asymp qT/N \\ (m,q)=1}}
   \int_{\mathbb{R}} V\left(y\right) e\left(T\varphi(y)\pm\frac{NXvy}{qQ}+\frac{mN}{q}y\right) \dd y
  \\
 \cdot \sum_{n_1|q} \sum_{n_2=1}^{\infty}
  \frac{A(n_2,n_1)}{n_2^{1/2}}
  S\left(-\bar{m},\pm n_2;\frac{q}{n_1}\right)
  e\left(\pm 2\frac{(n_1^2n_2 Q)^{1/2}}{q(Xv)^{1/2}}\right) w_{2}\left(\frac{(n_1^2n_2 N)^{1/2}}{(NXv/Q)^{3/2}}\right)
 \dd v,
\end{multline*}
which is equal to
\begin{multline}\label{eqn:S+-}
   \pm \frac{N^{3/2} X}{Q}
   \sum_{q\sim R}  \;\frac{1}{q^{3/2}}\; \sum_{\substack{m\asymp qT/N \\ (m,q)=1}} \int_{\mathbb{R}} V\left(y\right) e\left(T\varphi(y)+\frac{mN}{q}y\right)
   \\
 \cdot
 \sum_{n_1|q} \sum_{n_2 \asymp N^2 X^3/n_1^2 Q^3}
  \frac{A(n_2,n_1)}{n_2^{1/2}}
  S\left(-\bar{m},\pm n_2;\frac{q}{n_1}\right)
  \\
 \cdot
  \int_{\mathbb{R}} W_q\left(v\right)
  e\left(\pm\frac{NXvy}{qQ} \pm 2\frac{(n_1^2n_2 Q)^{1/2}}{q(Xv)^{1/2}}\right)
  w_{2}\left(\frac{(n_1^2n_2 N)^{1/2}}{(NXv/Q)^{3/2}}\right)
 \dd v \dd y ,
\end{multline}
Let
\[
  h_5(v) = \pm\frac{NXvy}{qQ} \pm 2\frac{(n_1^2n_2 Q)^{1/2}}{q(Xv)^{1/2}}.
\]
Then we have
\[
  h_5'(v) = \pm\frac{NXy}{qQ} \mp \frac{(n_1^2n_2 Q)^{1/2}}{q X^{1/2}} v^{-3/2},
\]
\[
  h_5''(v) = \pm  \frac{3(n_1^2n_2 Q)^{1/2}}{2q X^{1/2}} v^{-5/2}, \quad
  h_5^{(j)}(v) \asymp_j \frac{(n_1^2n_2 Q)^{1/2}}{q X^{1/2}}, \quad j\geq2.
\]
By repeated integration by parts we know that the $v$-integral is negligibly small unless
$\frac{NX}{RQ} \asymp  \frac{(n_1^2n_2 Q)^{1/2}}{R X^{1/2}}$, i.e., $n_1^2n_2 \asymp \frac{N^2 X^3}{Q^3}$, which we assume now.
The solution of $h_5'(v)=0$ is $v_0 = \frac{(n_1^2n_2)^{1/3} Q}{X (Ny)^{2/3}} $.
Note that  $h_5(v_0) = \pm 3 \frac{ (n_1^2 n_2 Ny)^{1/3} }{q}$.
By the stationary phase method (Lemma \ref{lemma:stationary_phase} (ii)), we have
\[
  \textrm{$v$-integral} =
    w\left(\frac{(n_1^2n_2)^{1/3} Q}{N^{2/3}X y^{2/3}}\right)
  \frac{(qQ)^{1/2}}{(NXy)^{1/2}}
   e\left(\pm 3 \frac{ (n_1^2 n_2 Ny)^{1/3} }{q}\right) + O(T^{-A}),
\]
where the function $w$ is some compactly supported $T^\varepsilon$-inert function depending on $q$ and $A$.
Thus it suffices to consider
\begin{equation}\label{eqn:S+-<<}
  \frac{N X^{1/2}}{Q^{1/2}}
   \sum_{q\sim R}  \;\frac{1}{q}\; \sum_{\substack{m\asymp qT/N \\ (m,q)=1}}
 \sum_{n_1|q} \sum_{n_2 \asymp N^2 X^3/n_1^2 Q^3}
  \frac{A(n_2,n_1)}{n_2^{1/2}}
  S\left(-\bar{m},\pm n_2;\frac{q}{n_1}\right)
  \mathcal W(m,n,q),
\end{equation}
where
\begin{equation}\label{eqn:W}
  \mathcal W(m,n_2,q) := \int_{\mathbb{R}} \frac{V\left(y\right)}{y^{1/2}} e\left(T\varphi(y)+\frac{mN}{q}y \pm 3 \frac{ (n_1^2 n_2 Ny)^{1/3} }{q}\right)
  w\left(\frac{(n_1^2n_2)^{1/3} Q}{N^{2/3}X y^{2/3}}\right)
  \dd y .
\end{equation}
Note that $\frac{ (n_1^2 n_2 N)^{1/3} }{q}\asymp NX/RQ \ll T^\varepsilon N/RQ \ll T^{1+\varepsilon}/Q$. By the second derivative test (see e.g. Huxley \cite[Lemma 5.1.3]{Huxley}), we get
\begin{equation}\label{eqn:W<<}
  \mathcal W(m,n_2,q) \ll T^{-1/2+\varepsilon}.
\end{equation}

\subsection{Applying the Cauchy--Schwarz inequality}\label{subsec:Cauchy}

By applying the Cauchy--Schwarz inequality tothe $(n_1,n_2)$-sum and \eqref{eqn:RS3}, we know that \eqref{eqn:S+-<<} is bounded by
\begin{equation}\label{eqn:S<<T}
   \frac{N X^{1/2}}{Q^{1/2}}  \frac{N X^{3/2}}{Q^{3/2}}
   \mathcal T^{\pm}(N;X,R;M,W) ^{1/2}   ,
\end{equation}
where $\mathcal T^{\pm} = \mathcal T^{\pm}(N;X,R;M,W)$ is given by
\begin{equation*}
  \sum_{n_1=1}^{\infty} \sum_{n_2=1}^{\infty}
   \frac{1}{n_2} W\left(\frac{n_1^2n_2}{M}\right)
   \Big|
   \sum_{\substack{R<q\leq  2R \\ n_1\mid q}} \;\frac{1}{q}\; \sum_{\substack{m\asymp qT/N \\ (m,q)=1}}
  S\left(-\bar{m},\pm n_2;\frac{q}{n_1}\right)
  \mathcal W(m,n_2,q)
    \Big|^2,
\end{equation*}
for some $M\asymp N^2X^3/Q^{3}$ and smooth functions $W$ with $\supp W\subset [1,2]$.

\begin{lemma}\label{lemma:T}
  We have
  \[
  \mathcal T^{\pm} \ll T^\varepsilon \left( \frac{R}{N}
  +  \frac{T Q^{6}}{N^{7/2}X^{5/2}} \right).
  \]
\end{lemma}

\begin{proof}[Proof of Theorem \ref{thm:3}]
By \eqref{eqn:S<<T}, we know the contribution to  $\mathscr S_1^\pm(N;X,R)$ is bounded by
\begin{equation}\label{eqn:S^pm<<l}
  T^\varepsilon \frac{N^{2}X^2}{Q^{2}}
  \left(\frac{R}{N} + \frac{T Q^{6}}{N^{7/2}X^{5/2}} \right)^{1/2}
  \ll T^\varepsilon \frac{N^{3/2}}{Q^{3/2}}
  + T^\varepsilon N^{1/4}Q T^{1/2}
  \ll  N^{3/4} T^{3/10+\varepsilon},
\end{equation}
by taking $Q = N^{1/2}/T^{1/5}$ if $T^{6/5}\leq N \leq T^{8/5-\varepsilon}$.
If $T^{8/5-\varepsilon}\leq N \leq T^{2}$, then we take $Q=NT^{\varepsilon-1}$ and show the contribution is bounded by $O(T^{-1/2}N^{5/4+\varepsilon})$.

For the case  $NX/RQ\gg T^{2\varepsilon}$ and $n_1^2n_2N/q^3 \ll T^\varepsilon$, or  the case  $NX/RQ\ll T^{2\varepsilon}$ and $n_1^2n_2N/q^3 \gg T^\varepsilon$, then the contribution to  $\mathscr S_1^\pm(N;X,R)$ is negligibly small by Lemma \ref{lemma:Psi}.

For the case  $NX/RQ\ll T^{2\varepsilon}$ and $n_1^2n_2N/q^3 \ll T^\varepsilon$, then by \eqref{eqn:RS3}, \eqref{eqn:V<<}, Lemma \ref{lemma:Psi}, and Weil bounds for Kloosterman sums, we can bound terms in \eqref{eqn:S1} trivially, showing the contribution to  $\mathscr S_1^\pm(N;X,R)$ is bounded by
\[
  \ll X\frac{N}{Q} R \frac{RT}{N} T^{-1/2} R^{1/2}
  \ll \frac{Q^{7/2}T^{1/2}}{N} \ll N^{3/4} T^{3/10+\varepsilon},
\]
by taking $Q = N^{1/2}/T^{1/5}$.
This completes the proof of Theorem \ref{thm:3}.
\end{proof}

\subsection{Applying Poisson again}\label{subsec:Poisson}

In this subsection we prove Lemma \ref{lemma:T}.
\begin{proof}[Proof of Lemma \ref{lemma:T}]
Opening the absolute square and interchanging the order of summations,
we get
\begin{equation}\label{eqn:T<<Sigma}
  \mathcal T^{\pm} \leq
   \sum_{n_1\leq 2R}
   \sum_{\substack{R<q\leq  2R \\ n_1\mid q}} \frac{1}{q}
   \sum_{\substack{m\asymp qT/N \\ (m,q)=1}}  \\
   \sum_{\substack{R<q'\leq  2R \\ n_1\mid q'}} \frac{1}{q'}
   \sum_{\substack{m'\asymp q'T/N \\ (m',q')=1}}
   | \Sigma |,
\end{equation}
where
\begin{equation*}
  \Sigma :=   \sum_{n_2=1}^{\infty}
   \frac{1}{n_2} W\left(\frac{n_1^2n_2}{M}\right)
  S\left(-\bar{m},\pm n_2;\frac{q}{n_1}\right)
  S\left(\bar{m}',\mp n_2;\frac{q'}{n_1}\right)
  \mathcal W(m,n_2,q) \overline{ \mathcal W(m',n_2,q')}.
\end{equation*}

Breaking  the $n_2$-sum modulo $qq'/n_1^2$, we obtain
\begin{multline*}
  \Sigma =
  \sum_{b \mod qq'/n_1^2}  \sum_{n_2\in\mathbb{Z}}
   \frac{W\left(\frac{n_1^2 b+n_2qq'}{M}\right)}{b+n_2qq'/n_1^2}
  S\left(-\bar m,\pm b;\frac{q}{n_1}\right)
  S\left(\bar m',\mp b;\frac{q'}{n_1}\right)
    \\
  \cdot
  \mathcal W(m,b+n_2qq'/n_1^2,q) \overline{ \mathcal W(m',b+n_2qq'/n_1^2,q')}.
\end{multline*}
Applying the Poisson summation formula to the $n_2$-sum, we get
\begin{multline*}
  \Sigma =  \sum_{b \mod qq'/n_1^2}
   S\left(-\bar m,\pm b;\frac{q}{n_1}\right)
  S\left(\bar m',\mp b;\frac{q'}{n_1}\right)
   \sum_{n\in\mathbb{Z}} \int_{\mathbb{R}}
   \frac{W\left(\frac{n_1^2b+ \xi qq'}{M}\right)}{b+ \xi qq'/n_1^2}
    \\
  \cdot
  \mathcal W(m,b+\xi qq'/n_1^2,q)
  \overline{ \mathcal W(m',b+\xi qq'/n_1^2,q')}
    e(-n \xi) \dd \xi.
\end{multline*}
Making a change of variables $\frac{n_1^2b+\xi qq'}{M} \mapsto \xi$, we have
\begin{equation*}
  \Sigma = \frac{n_1^{2}}{qq'}
   \sum_{n\in\mathbb{Z}} \mathfrak{C}(n,m,m';q/n_1,q'/n_1)
    \mathfrak{J}(n,m,m';q,q') ,
\end{equation*}
where
\[
   \mathfrak{C}(n,m,m';\hat q,\hat q')
   := \sum_{b \mod qq'/n_1^2}
   S\left(-\bar m,\pm b;\hat{q} \right)
  S\left(\bar m',\mp b;\hat q'\right)
    e\left(\frac{n b}{\hat q \hat q'}\right),
\]
and
\begin{equation*}
   \mathfrak{J}(n,m,m';q,q'):= \int_{-\infty}^{\infty}
   \frac{W\left(\xi\right)}{\xi}
  \mathcal W\Big(m,\frac{M}{n_1^2}\xi,q\Big)
  \overline{ \mathcal W\Big(m',\frac{M}{n_1^2}\xi,q'\Big)}
    e\left(-n \frac{M\xi}{qq'} \right) \dd \xi.
\end{equation*}

The character sum $\mathfrak{C}$ has already appeared in many places, see e.g. \cite[Lemma 11]{Munshi2015circleIII}. We have the following bounds.
\begin{lemma}
  We have
  \[
    \mathfrak{C}(n,m,m';\hat q,\hat q')
    \ll \hat{q} \hat{q}' (\hat{q},\hat{q}',n).
  \]
  Moreover if $n=0$ then we get $\mathfrak{C}(n,m,m';\hat q,\hat q') =0$ unless $\hat{q}=\hat{q}'$, in which case we have
  \[
    \mathfrak{C}(n,m,m';\hat q,\hat q')
    \ll \hat{q}^2 (\hat{q},m-m').
  \]
\end{lemma}

We also need bounds for the integrals $\mathfrak{J}(n,m,m';q,q')$.
We will prove the following lemma.
\begin{lemma}\label{lemma:J}
  We have
  \begin{itemize}
    \item [(i)]  If $n \geq C RQ^2/NX^2$ for some large constant $C>0$, then $\mathfrak{J}(n,m,m';q,q') \ll n^{-2} T^{-2020}$.
    \item [(ii)] Assume  $V$ is $P$-inert for some $P\ll T^{\eta}$ with some positive constant  $\eta<1/10$. Assume $Q\geq T^{1/3+\varepsilon} P^{2/3}$. If $T^\varepsilon R^2Q^3/N^2X^3  \ll n \ll RQ^2/NX^2$, then $$\mathfrak{J}(n,m,m';q,q') \ll T^{-1} (nM/R^2)^{-1/2}.$$
    \item [(iii)] If $n \ll T^\varepsilon R^2Q^3/N^2X^3$, then $\mathfrak{J}(n,m,m';q,q') \ll T^{-1+\varepsilon}$.
  \end{itemize}
\end{lemma}

\begin{proof}
If $n \ll T^\varepsilon R^2Q^3/N^2X^3$, then by \eqref{eqn:W<<}, we get $\mathfrak{J}(n,m,m';q,q') \ll T^{-1+\varepsilon}$. This proves Lemma \ref{lemma:J} (iii).

\medskip

If $n \geq C RQ^2/NX^2$ for some large constant $C>0$, then we have $\frac{Mn}{qq'} \geq C' \frac{NX}{RQ}$ for some large constant $C'>0$.
By \eqref{eqn:W}, we have
\begin{multline*}
   \mathfrak{J}(n,m,m';q,q')=
  \int_{\mathbb{R}} \frac{V\left(y\right)}{y^{1/2}}  e\left(T\varphi(y)+\frac{mN}{q}y \right)
  \int_{\mathbb{R}} \frac{V\left(y'\right)}{y'^{1/2}} e\left(-T\varphi(y')-\frac{m'N}{q'}y' \right)    \\
  \cdot
   \int_{-\infty}^{\infty}
   \frac{W\left(\xi\right)}{\xi} w\left(\frac{(M \xi)^{1/3} Q}{N^{2/3}X y^{2/3}}\right) w\left(\frac{(M \xi)^{1/3} Q}{N^{2/3}X y'^{2/3}}\right)   \\
  \cdot e\left(\pm 3 \frac{ (M Ny\xi)^{1/3} }{q} \mp 3 \frac{ (M Ny'\xi)^{1/3} }{q'} -n \frac{M\xi}{qq'} \right) \dd \xi
   \dd y \dd y'.
\end{multline*}
Note that by $M\asymp N^2X^3/Q^{3}$ we have $\frac{(M Ny)^{1/3} }{q} \ll \frac{NX}{RQ}$.
Let
\[
  h_6(\xi)= \pm 3 \frac{ (M Ny\xi)^{1/3} }{q} \mp 3 \frac{ (M Ny'\xi)^{1/3} }{q'} -n \frac{M\xi}{qq'}.
\]
Then we have
\[
  h_6'(\xi)= \pm \frac{ (M Ny)^{1/3} }{q} \xi^{-2/3} \mp \frac{ (M Ny')^{1/3} }{q'} \xi^{-2/3}
   -n \frac{M}{qq'} \gg \frac{NX}{RQ},
\]
\[
  h_6''(\xi)= \mp \frac{2 (M Ny)^{1/3} }{3q} \xi^{-5/3} \pm \frac{2 (M Ny')^{1/3} }{3 q'} \xi^{-5/3} \ll \frac{NX}{RQ},
\]
\[
  h_6^{(j)}(\xi) \ll_j \frac{NX}{RQ}, \quad j\geq2.
\]
Note that the weight function $\frac{W\left(\xi\right)}{\xi} w\left(\frac{(M \xi)^{1/3} Q}{N^{2/3}X y^{2/3}}\right) w\left(\frac{(M \xi)^{1/3} Q}{N^{2/3}X y'^{2/3}}\right)$ in $\xi$ is $T^\varepsilon$-inert.  By  Lemma \ref{lemma:stationary_phase} (i), we prove Lemma \ref{lemma:J} (i).

\medskip
Now we assume $Q\geq T^{1/3}$ and $T^\varepsilon R^2Q^3/N^2X^3  \ll n \ll RQ^2/NX^2$.
We also assume $V$ is $P$-inert and $\varphi(u)=u^\beta$.
We write $\mathcal W\left(m,\frac{M}{n_1^2}\xi,q\right)$ as
\begin{equation*}
   \int_{\mathbb{R}} V_1\left(y\right) e\left(T y^\beta +\frac{mN}{q}y \pm 3 \frac{ (M N \xi y)^{1/3} }{q}\right)
   \dd y ,
\end{equation*}
where $V_1(y) = \frac{V\left(y\right)}{y^{1/2}} w\left(\frac{(M \xi)^{1/3} Q}{N^{2/3}X y^{2/3}}\right) $ is a $P$-inert function.
Let
\[
  A=-mN/q, \quad
  B=\pm \frac{ (M N \xi)^{1/3} }{q}, \quad
  \textrm{and} \quad
  h_7(y)=T y^\beta -Ay +3B y^{1/3}.
\]
Note that  $B \asymp NX/RQ \ll T^{1+\varepsilon}/Q \leq T^{2/3+\varepsilon}$.
We have
\[
  h_7'(y) =  \beta T y^{\beta-1} -A + B y^{-2/3} \gg T + |A|,
\]
unless  $A\asymp T$. Note that $h_7^{(j)}(y) \asymp_{\beta,j} T$ for $y\asymp 1$.
Hence by Lemma \ref{lemma:stationary_phase} (i), $\mathcal{W}$ is negligibly small unless $A\asymp T$.
Since $Q\geq T^{1/3}$, we have $B^4/A^3 =O(T^\varepsilon/Q)$.
The solution to $h_7'(y)=0$, i.e., $T\beta y^{\beta-1}-A+By^{-2/3}=0$ in the support of $V_1$ will be $y_0+y_1+y_2+y_3$ where
\[
  y_0 = \left(\frac{A}{T\beta}\right)^{\frac{1}{(\beta-1)}}, \
  y_1 
  = \frac{y_0^{1/3}}{1-\beta} \frac{B}{A} ,
\]
\[
  y_2 
  = \frac{(2-3\beta) }{6 (\beta-1)^2 y_0^{1/3}} \frac{B^2}{A^2},
 \quad
  y_j \ll \left(\frac{B}{A}\right)^j  , \quad 0\leq j\leq3.
\]
Here $y_0$ satisfies that  $T\beta y_0^{\beta-1}-A=0$, $y_1$ satisfies that
$T\beta y_0^{\beta-1} (1+(\beta-1)y_1/y_0) -A+By_0^{-2/3}=0$, and $y_2$ satisfies that
$T\beta y_0^{\beta-1} \big(1+(\beta-1)\frac{y_1+y_2}{y_0}+\frac{(\beta-1)(\beta-2)}{2} \frac{y_1^2}{y_0^2}\big) - A + By_0^{-2/3} (1-\frac{2}{3}\frac{y_1}{y_0})=0$.
Note that
\begin{multline*}
  0 = T\beta (y_0+y_1+y_2+y_3)^{\beta-1}-A+B(y_0+y_1+y_2+y_3)^{-2/3} \\
  =
  T\beta y_0^{\beta-1} \left( 1+(\beta-1)\frac{y_1+y_2+y_3}{y_0}+\frac{(\beta-1)(\beta-2)}{2} \frac{y_1^2}{y_0^2} + O((B/A)^3)\right) \\
  - A
  + By_0^{-2/3} \left(1-\frac{2}{3}\frac{y_1}{y_0} + O((B/A)^2)\right) \\
  = T\beta y_0^{\beta-1}  (\beta-1)\frac{y_3}{y_0}  +
    O(B^3/A^2).
\end{multline*}
This give $y_3 \ll (B/A)^3$.
Hence
\begin{multline*}
  h_7(y_0+y_1+y_2+y_3)  \\
  = T y_0^\beta (1+y_1/y_0+y_2/y_0+y_3/y_0)^{\beta} -A(y_0+y_1+y_2+y_3) +3B (y_0+y_1+y_2)^{1/3} +O(B^4/A^3) \\
  = \left(\frac{1}{\beta} -1\right)Ay_0
  + 3y_0^{1/3} B
  + \frac{1}{2(1-\beta) y_0^{1/3}} \frac{B^2}{A}
  - \frac{\beta}{6(\beta-1)^2 y_0} \frac{B^3}{A^2} + O(T^\varepsilon/Q).
\end{multline*}
Note that $h_7''(y)=\beta(\beta-1) T y^{\beta-2} - \frac{2}{3} B y^{-5/3}$. We have
\[
  h_7''(y_0+y_1+y_2+y_3) = \beta(\beta-1) T y_0^{\beta-2} (1+O(B/A)).
\]
By the stationary phase method (see Lemma \ref{lemma:stationary_phase} (ii)), we get
\[
  \mathcal W\Big(m,\frac{M}{n_1^2}\xi,q\Big) = \frac{ e\left(h_7(y_0+y_1+y_2+y_3)\right)}{\sqrt{T}}
  w_1(y_0+y_1+y_2+y_3) (1+O(B/A)) + O(T^{-2020}),
\]
for certain $P$-inert function $w_1$. Note that $y_1+y_2+y_3 \ll B/A \ll T^\varepsilon /Q$, we have $w_1(y_0+y_1+y_2+y_3)=w_1(y_0)+O(P T^\varepsilon /Q)$. Hence we obtain
\[
  \mathcal W\Big(m,\frac{M}{n_1^2}\xi,q\Big) = \frac{1}{\sqrt{T}}
  e\left(3y_0^{1/3} B
  + \frac{1}{2(1-\beta) y_0^{1/3}} \frac{B^2}{A}
  - \frac{\beta}{6(\beta-1)^2 y_0} \frac{B^3}{A^2} \right) w_2(y_0)
  + O\left( \frac{PT^{\varepsilon}}{T^{1/2} Q}\right),
\]
where $w_2(y_0)=e((\frac{1}{\beta} -1)Ay_0)w_1(y_0)$.
Hence
\begin{multline*}
   \mathfrak{J}(n,m,m';q,q') \ll
   \bigg| \frac{1}{T} \int_{-\infty}^{\infty}
   \frac{W\left(\xi\right)}{\xi}
  e\left(  3y_0^{1/3} B
  + \frac{1}{2(1-\beta) y_0^{1/3}} \frac{B^2}{A}
  - \frac{\beta}{6(\beta-1)^2 y_0} \frac{B^3}{A^2} \right)
  \\ \cdot
  e\left( -3y_0'^{1/3} B'
  - \frac{1}{2(1-\beta) y_0'^{1/3}} \frac{B'^2}{A'}
  + \frac{\beta}{6(\beta-1)^2 y_0'} \frac{B'^3}{A'^2} \right)
    e\left(-n \frac{M\xi}{qq'} \right) \dd \xi \bigg|
  + \frac{P T^\varepsilon}{TQ},
\end{multline*}
where $A',B',y_0'$ are defined the same as $A,B,y_0$ but with $m,q$ replaced by $m',q'$.
Note that we assume $n \ll RQ^2/NX^2$ and $Q\geq T^{1/3+\varepsilon}P^{2/3}$, hence we have
\[
  (nM/R^2)^{1/2} \ll (NX/RQ)^{1/2} \ll T^{1/2+\varepsilon}/Q^{1/2} \ll Q/P,
\]
that is,
\[
  T^\varepsilon
   P T^{-1}Q^{-1} \ll T^{-1} (nM/R^2)^{-1/2}.
\]

Now we deal with the $\xi$-integral.
Let
\begin{multline*}
  h_8(\xi) =  3y_0^{1/3} B
  + \frac{1}{2(1-\beta) y_0^{1/3}} \frac{B^2}{A}
  - \frac{\beta}{6(\beta-1)^2 y_0} \frac{B^3}{A^2}
  \\  -3y_0'^{1/3} B'
  - \frac{1}{2(1-\beta) y_0'^{1/3}} \frac{B'^2}{A'}
  + \frac{\beta}{6(\beta-1)^2 y_0'} \frac{B'^3}{A'^2}
   -n \frac{M\xi}{qq'}.
\end{multline*}
Recall that $B=\pm \frac{ (M N \xi)^{1/3} }{q}$ and $B'=\pm \frac{ (M N \xi)^{1/3} }{q'}$.
By $R\gg N/T$ and $NX/RQ\gg T^\varepsilon$, we get $X\gg QT^{\varepsilon-1}$. So we obtain for $|n|\geq1$,
\[
  nM/qq' \gg \frac{N^2 X^3}{R^2 Q^3 } \gg T^\varepsilon \frac{N^2X^2}{R^2Q^2 T} \asymp T^\varepsilon B^2/T.
\]
Hence
\begin{multline*}
  h_8'(\xi) = \pm y_0^{1/3} \frac{ (M N)^{1/3} }{q} \xi^{-2/3}
  + \frac{1}{2(1-\beta) y_0^{1/3}} \frac{1}{A} \frac{ (M N)^{2/3} }{q^2} \xi^{-1/3}
  \mp \frac{\beta}{6(\beta-1)^2 y_0} \frac{1}{A^2} \frac{M N}{q^3}
  \\  \mp y_0'^{1/3} \frac{ (M N)^{1/3} }{q'} \xi^{-2/3}
  - \frac{1}{2(1-\beta) y_0'^{1/3}} \frac{1}{A'} \frac{ (M N)^{2/3} }{q'^2} \xi^{-1/3}
  \pm \frac{\beta}{6(\beta-1)^2 y_0'} \frac{1}{A'^2} \frac{M N}{q'^3}
   -n \frac{M}{qq'} \\
   = \pm \left( y_0^{1/3} \frac{ (M N)^{1/3} }{q} -
    y_0'^{1/3} \frac{ (M N)^{1/3} }{q'} \right) \xi^{-2/3}
   -n \frac{M}{qq'} + O\left(\frac{B^2}{T}\right) ,
\end{multline*}
and
\[
  h_8''(\xi) = \mp \frac{2}{3} \left( y_0^{1/3} \frac{ (M N)^{1/3} }{q} -
    y_0'^{1/3} \frac{ (M N)^{1/3} }{q'} \right) \xi^{-5/3}
    + O\left(\frac{B^2}{T}\right).
\]
Since $nM/qq' \gg T^\varepsilon$, by repeated integration by parts we know that $\mathfrak{J}(n,m,m';q,q')$ is negligibly small unless
\begin{equation}\label{eqn:restriction}
  \bigg( y_0^{1/3} \frac{ (M N)^{1/3} }{q} -
    y_0'^{1/3} \frac{ (M N)^{1/3} }{q'} \bigg)
   \asymp \frac{nM}{R^2},
\end{equation}
in which case we have $h_8^{(j)}(\xi) \asymp_j nM/R^2$ for $j\geq2$,
and hence by Lemma \ref{lemma:stationary_phase} (ii) we get  $\mathfrak{J}(n,m,m';q,q') \ll T^{-1} (nM/R^2)^{-1/2}$. This completes the proof of Lemma \ref{lemma:J}.
\end{proof}

Now we are ready to bound $\Sigma$. Depending on $n=0$, $n$ small, and $n$ large, we have
\[
  \Sigma \ll \Sigma_0 + \Sigma_{1} + \Sigma_2 + T^{-A},
\]
where
\[
  \Sigma_0 \ll \frac{n_1^{2}}{R^2} \delta_{q=q'} \frac{R^2}{n_1^2} (q/n_1,m-m')  T^{-1+\varepsilon}
   \ll T^{-1+\varepsilon} \delta_{q=q'}  (q/n_1,m-m'),
\]
\[
  \Sigma_1 \ll \frac{n_1^{2}}{R^2 }
   \sum_{1\leq n \ll  \frac{R^2 Q^3}{N^2 X^3} T^\varepsilon}
   \frac{R^2}{n_1^2} (q/n_1,q'/n_1,n)
   T^{-1+\varepsilon}
   \ll T^{-1+\varepsilon}
   \sum_{1\leq n \ll  \frac{R^2 Q^3}{N^2 X^3} T^\varepsilon}
   (q/n_1,q'/n_1,n) ,
\]
and
\begin{multline*}
  \Sigma_2
   \ll \frac{n_1^{2}}{R^2 }
   \sum_{\frac{R^2 Q^3}{N^2 X^3} T^\varepsilon \ll n
   \ll \frac{R Q^2 }{NX^2}}  \frac{R^2}{n_1^2} (q/n_1,q'/n_1,n)
   \frac{RQ^{3/2}}{n^{1/2}NX^{3/2}} \frac{1}{T} \\
    \ll \frac{1}{T}  \sum_{\frac{R^2 Q^3}{N^2 X^3} T^\varepsilon \ll n \ll \frac{R Q^2 }{NX^2}} (q/n_1,q'/n_1,n)
   \frac{RQ^{3/2}}{n^{1/2}NX^{3/2}}.
\end{multline*}
Denote the corresponding contribution from $\Sigma_j$ to $\mathcal{T}^\pm$ by $\mathcal{T}_j$. We first bound $\mathcal{T}_0$. By \eqref{eqn:T<<Sigma}, we get
\begin{equation*}
  \mathcal T_0 \ll \frac{1}{R^2} \sum_{n_1\leq 2R}
   \sum_{\substack{R<q\leq  2R \\ n_1\mid q}}
   \sum_{\substack{m\asymp qT/N \\ (m,q)=1}}  \\
   \sum_{\substack{m'\asymp qT/N \\ (m',q)=1}}
   T^{-1+\varepsilon} (q/n_1,m-m').
\end{equation*}
Arguing depending on whether $m=m'$ or not, we have
\begin{multline}\label{eqn:T0}
   \mathcal T_0 \ll T^{-1+\varepsilon} \frac{1}{R}  \sum_{n_1\leq 2R}
       \frac{1}{n_1}
   \sum_{\substack{R<q\leq  2R \\ n_1\mid q}}
   \sum_{\substack{m\asymp qT/N \\ (m,q)=1}}1
   \\
   + T^{-1+\varepsilon} \frac{1}{R^2 }  \sum_{n_1\leq 2R}
   \sum_{\substack{R<q\leq  2R \\ n_1\mid q}}
   \sum_{\substack{m\asymp qt/N \\ (m,q)=1}}
   \sum_{1\leq |h| \ll \frac{RT}{N}}
    (q/n_1,h)
    \ll \frac{R}{N} T^{\varepsilon}  .
\end{multline}
Next, we consider $\mathcal{T}_1$. By \eqref{eqn:T<<Sigma}, we obtain
\begin{multline}\label{eqn:T1}
  \mathcal T_1
  \ll  \frac{1}{R^2} \sum_{n_1\leq 2R}
   \sum_{\substack{R<q\leq  2R \\ n_1\mid q}}
   \sum_{\substack{m\asymp qt/N \\ (m,q)=1}}
   \sum_{\substack{R<q'\leq  2R \\ n_1\mid q'}}
   \sum_{\substack{m'\asymp q't/N \\ (m',q')=1}}
   T^{-1+\varepsilon}
   \sum_{1\leq n \ll  \frac{R^2 Q^3}{N^2 X^3} T^\varepsilon}
   (q/n_1,q'/n_1,n)
\\
   \ll  T^\varepsilon \frac{R^4 T Q^3}{N^4 X^3}.
\end{multline}
Finally, we treat $\mathcal{T}_2$, getting
\begin{multline}\label{eqn:T2}
  \mathcal T_2
  \ll  \frac{1}{R^2 } \sum_{n_1\leq 2R}
   \sum_{\substack{R<q\leq  2R \\ n_1\mid q}}
   \sum_{\substack{m\asymp qt/N \\ (m,q)=1}}
   \sum_{\substack{R<q'\leq  2R \\ n_1\mid q'}}
   \sum_{\substack{m'\asymp q't/N \\ (m',q')=1}}
    \\ \cdot
  \frac{1}{T}  \sum_{\frac{R^2 Q^3}{N^2 X^3} T^\varepsilon \ll n \ll \frac{R Q^2 }{NX^2}} (q/n_1,q'/n_1,n)
   \frac{RQ^{3/2}}{n^{1/2}NX^{3/2}}
   \ll  \frac{T R^{7/2}Q^{5/2}}{N^{7/2}X^{5/2}} .
\end{multline}
Note that
$T^\varepsilon \frac{R^4 T Q^3}{N^4 X^3} \ll \frac{T R^{7/2}Q^{5/2}}{N^{7/2}X^{5/2}}$ if $NX/RQ\gg T^{2\varepsilon}$. Combining \eqref{eqn:T0}, \eqref{eqn:T1} and \eqref{eqn:T2}, we complete the proof of Lemma \ref{lemma:T}.
\end{proof}


%
%

\section*{Acknowledgements}
The author is grateful to the referee for his/her very helpful comments and suggestions.
He  wants to thank Prof. Jianya Liu and Ze\'{e}v Rudnick for their help and encouragements.
He would like to thank Yongxiao Lin and  Qingfeng Sun for  many discussions on  \cite{FI2005,LinSun}.
The author started to try seriously to solve the Rankin--Selberg problem when he was a postdoc at Tel Aviv University. 
He likes to thank Tel Aviv Unviersity for providing a nice work environment.


\end{document}